\documentclass[reqno,letterpaper,11pt]{article}

\usepackage[draft=false]{hyperref}

\usepackage{times}
\usepackage[mathcal]{euscript}

\usepackage{amsmath}
\usepackage{amsfonts,amssymb,amsmath,latexsym,wasysym,mathrsfs,bbm,stmaryrd}
\usepackage[all]{xy}
\usepackage[usenames]{color}
\usepackage{epsfig}

\usepackage{amsthm}

\usepackage{thmtools}
\usepackage{thm-restate}
\declaretheorem[numberwithin=section]{theorem}
\declaretheorem[sibling=theorem]{proposition}
\declaretheorem[sibling=theorem]{definition}
\declaretheorem[sibling=theorem]{corollary}
\declaretheorem[sibling=theorem]{lemma}

\declaretheorem[sibling=theorem,style=remark]{remark}

\numberwithin{equation}{section} 

\def\R{\mathbb R}
\def\C{\mathbb C}

\def\N{\mathbb N}

\def\P{\mathbb P}
\def\E{\mathbb E}

\def\deg{\mathrm{deg}}

\def\1{\mathbbm 1}

\def\ex{\varepsilon}

\def\M{\mathbb{M}}
\def\U{\mathbb{U}}
\def\GL{\mathbb{GL}}
\def\A{\mathscr{A}}
\def\EX{\mathscr{E}}
\def\PP{\mathscr{P}}
\def\D{\mathcal{D}}
\def\L{\mathcal{L}}

\def\tensor{\otimes}

\def\del{\partial}

\newcommand{\tr}{\mathrm{tr}}
\newcommand{\End}{\mathrm{End}}
\newcommand{\Tr}{\mathrm{Tr}}
\newcommand{\mx}[1]{\mathbf{#1}}
\newcommand{\vertiii}[1]{{\left\vert\kern-0.25ex\left\vert\kern-0.25ex\left\vert #1 
    \right\vert\kern-0.25ex\right\vert\kern-0.25ex\right\vert}}

\begin{document}

\title{Fluctuations of Brownian Motions on $\mathbb{GL}_N$}
\author{Guillaume C\'ebron\thanks{Supported by the ERC advanced grant "Noncommutative distributions in free probability"} \\
Fachrichtung Mathematik \\
Saarland University \\
66123, Saarbr\"{u}cken, Germany \\
\texttt{cebron@math.uni-sb.de}
\and
Todd Kemp\thanks{Supported by NSF CAREER Award DMS-1254807} \\
Department of Mathematics \\
University of California, San Diego \\
La Jolla, CA 92093-0112 \\
\texttt{tkemp@math.ucsd.edu}
}

\date{\today}% \quad {\em File:} \jobname{.tex}}

\maketitle

\begin{abstract} We consider a two parameter family of unitarily invariant diffusion processes on the general linear group $\GL_N$ of $N\times N$ invertible matrices, that includes the standard Brownian motion as well as the usual unitary Brownian motion as special cases.  We prove that all such processes have Gaussian fluctuations in high dimension with error of order $O(\frac1N)$; this is in terms of the finite dimensional distributions of the process under a large class of test functions known as trace polynomials.  We give an explicit characterization of the covariance of the Gaussian fluctuation field, which can be described in terms of a fixed functional of three freely independent free multiplicative Brownian motions.  These results generalize earlier work of L\'evy and Ma\"ida, and Diaconis and Evans, on unitary groups. Our approach is geometric, rather than combinatorial.  \end{abstract}

\tableofcontents

\section{Introduction}

This paper is concerned with the fluctuations of of Brownian motions on the general linear groups $\GL_N = \mathrm{GL}(N,\C)$ when the dimension $N$ tends to infinity.

Let $\M_N$ denote the space of $N\times N$ complex matrices. A {\em random matrix ensemble} or {\em model} is a sequence of random variables $(B^N)_{N\geq 1}$ such that $B^N\in \M_N$. The first  phenomenon typically studied is the convergence in noncommutative distribution (cf.\ Section \ref{s.fmbm}) of $B^N$, meaning that for each noncommutative polynomial $P$ in two variables, we ask for convergence of $\E[\tr(P(B^N,{B^N}^*))]$, where $\tr$ is the normalized trace (so that $\tr(I_N)=1$).  In the special case that $B^N=(B^N)^\ast$ is self-adjoint, this is morally (and usually literally) equivalent to weak convergence in expectation of the empirical spectral distribution of $B^N$: the random probability measure placing equal masses at each of the random eigenvalues of the matrix.  The prototypical example here is Wigner's semicircle law \cite{Wigner1958}: if $B^N$ is a Wigner ensemble (meaning it is self-adjoint and the upper triangular entries are i.i.d.\ normal random variables with mean $0$ and variance $\frac1N$) then as $N\to\infty$ the empirical spectral distribution converges to $\frac{1}{2\pi}\sqrt{4-x^2}dx$.  In fact, the weak convergence is not only in expectation but almost sure.

For non-self-adjoint (and more generally non-normal) ensembles that cannot be characterized by their eigenvalues, the noncommutative distribution is the right object to consider.  As with Wigner's law, in most cases, we have the stronger result of almost sure convergence of the random variable $\tr(P(B^N,{B^N}^*))$ to its mean. It is therefore natural to ask for the corresponding central limit theorem: what is the rate of convergence to the mean, and what is the noise profile that remains?  More precisely, consider the random variables
\[ \tr(P(B^N,{B^N}^*))-\E[\tr(P(B^N,{B^N}^*))] \]
for each noncommutative polynomial $P$; these are known as the {\em fluctuations}.  The question is: what is their order of magnitude, and when appropriately renormalized, what is their limit as $N\to\infty$? The standard scaling for this kind of central limit theorem in random matrix theory is well-known to be $\frac{1}{N}$ instead of the classical $\frac{1}{\sqrt{N}}$ (see the fundamental work of Johansson~\cite{Johansson1998}). Thus far, it was known that
$$N\left(\tr(P(B^N,{B^N}^*))-\E[\tr(P(B^N,{B^N}^*))]\right)$$
is asymptotically Gaussian when
\begin{itemize}
\item $B^N$ is a Wigner random matrix~\cite{Cabanal-Duvillard2001};
\item $B^N$ is a unitary random matrix whose distribution is the Haar measure~\cite{Diaconis2001};
\item $B^N$ is a unitary random matrix arising from a Brownian motion on the unitary group~\cite{LevyMaida2010} or the orthogonal group~\cite{Dahlqvist2014}.
\end{itemize}

\begin{remark} The existence of Gaussian fluctuations of a random matrix model is sometimes referred to as a {\em second order distribution}; cf. \cite{MingoSpeicher2007,MingoSpeicher2006}, in which the authors gave the corresponding diagrammatic combinatorial theory of fluctuations.  The similar but more complicated combinatorial approach to fluctuations for Haar unitary ensembles was done in \cite{Collins2003,CollinsSniady2006}, where is is known as Weingarten calculus.  The recent work of Dahlqvist \cite{Dahlqvist2013} follows these ideas to provide the combinatorial framework for the finite-time heat kernels on classical compact Lie groups.  Our present approach is geometric, rather than combinatorial.
\end{remark}

Our main result is of this type, when $B^N$ is sampled from a two-parameter family of random matrix ensembles that may rightly be called {\em Brownian motions} on $\GL_N$. Fix $r,s>0$, and following~\cite{Kemp2013b}, we will define (in Section \ref{not}) an {\em $(r,s)$-Brownian motion} $(B^N_{r,s}(t))_{t\geq 0}$ on $\GL_N$ for each dimension $N>0$.  This family encompasses the two most well-studied Brownian motions on invertible matrices: the canonical Brownian motion $G^N(t)\equiv B^N_{\frac12,\frac12}(t)$ on $\GL_N$, and the canonical Brownian motion $U^N(t)\equiv B_{1,0}^N(t)$ on the unitary group $\U_N$.  These processes are given as solutions to matrix stochastic differential equations
\[ dG^N(t) = G^N(t)\,dZ^N(t), \qquad dU^N(t) = i U^N(t)\,dX^N(t) - \frac12U^N(t)\,dt \]
where the entries of $Z^N(t)$ are i.i.d.\ complex Brownian motions of variance $\frac{t}{N}$, and $X^N(t)=\sqrt{2}\Re(Z^N(t))$.  The study of the convergence in noncommutative distribution of $U^N(t)$ was completed by Biane~\cite{Biane1997}, and the case of $G^N(t)$ (for fixed $t>0$) was completed by the first author~\cite{Cebron2013}; the second author introduced the general processes $B^N_{r,s}(t)$ in~\cite{Kemp2013a,Kemp2013b} and proved they converge (as processes) a.s.\ in noncommutative distribution to the relevant free analog, {\em free multiplicative $(r,s)$-Brownian motion} (cf.\ Section \ref{s.fmbm}).  This naturally leads to the question of fluctuations of all these processes, which we answer in our Main Theorem~\ref{main} and Corollary \ref{c.main}.  We summarize a slightly simplified form of the result here as Theorem \ref{t.main.summ}.

\begin{theorem} \label{t.main.summ} Let $(B_{r,s}^N(t))_{t\ge 0}$ be an $(r,s)$-Brownian motion on $\GL_N$.  Let $n\in\N$ and $t_1,\ldots,t_n\ge 0$; set $\mx{T}=(t_1,\ldots,t_n)$, and let $B_{r,s}^N(\mx{T}) = (B_{r,s}^N(t_1),\ldots,B_{r,s}^N(t_n))$.  Let $P_1,\ldots,P_k$ be noncommutative polynomials in $2n$ variables, and define the random variables
\begin{equation}\label{e.fluctuations} X_j = N[\tr (P_j(B^N_{r,s}(\mx{T}),B^N_{r,s}(\mx{T})^\ast))-\E\tr(P_j(B^N_{r,s}(\mx{T}),B^N_{r,s}(\mx{T})^\ast))], \qquad 1\le j\le k. \end{equation}
Then, as $N\to\infty$, $(X_1,\ldots,X_k)$ converges in distribution to a multivariate centered Gaussian.
 \end{theorem}
As mentioned, Theorem \ref{t.main.summ} generalizes the main theorem \cite[Theorem 2.6]{LevyMaida2010} to general $r,s>0$ from the $(r,s)=(1,0)$ case considered there. In fact, even when $(r,s)=(1,0)$ this is a significant generalization, as the fluctuations proved in \cite{LevyMaida2010} were for a single time $t$ -- i.e.\ for a heat-kernel distributed random matrix -- while we prove the optimal result for the full process -- i.e.\ for all finite-dimensional distributions.

\begin{remark} \begin{itemize} \item[(1)] In fact, \cite[Theorem 8.2]{LevyMaida2010} does give a partial generalization to multiple times, in the sense that the argument of $P_j$ in \eqref{e.fluctuations} is allowed to depend on $B^N_{1,0}(t_j)$ for a $j$-dependent time; however, it must still be a function of Brownian motion at a single time.  Our generalization allows full consideration of all finite-dimensional distributions.
\item[(2)] To be fair, \cite{LevyMaida2010} yields Gaussian fluctuations for a larger class of test functions.  In the case of a single time $t$, the random matrix $U^N(t)$ is normal, and hence ordinary functional calculus makes sense; the fluctuations in \cite{LevyMaida2010} extend beyond polynomial test functions to $C^1$ functions with Lipschitz derivative on the unit circle.  Such a generalization is impossible for the generically non-normal matrices in $\GL_N$.\end{itemize} \end{remark}

Theorem \ref{main} actually gives a further generalization of Theorem \ref{t.main.summ}, as the class of test functions is not just restricted to traces of polynomials, but the much larger algebra of {\em trace polynomials}, cf.\ Section \ref{s.C{J}}.  That is, we may consider more general functions of the form $Y_j = \tr(P_j^1)\cdots \tr(P_j^n)$ (or linear combinations thereof); then the result of Theorem \ref{t.main.summ} applies to the fluctuations $X_j = N[Y_j-\E(Y_j)]$ as well.

% states that $N\left(\tr(P(B^N_{r,s}(t),{B^N}^*_{r,s}(t)))-\E[\tr(P(B^N_{r,s}(t),{B^N}^*_{r,s}(t)))]\right)$ is asymptotically Gaussian for each noncommutative polynomial $P$ and for each time $t>0$.
%
%In fact, Theorem~\ref{main} is more general with respect to two facts. Firstly, it deals with a whole family of random matrices arising from $(r,s)$-Brownian motions stopped at different times. Secondly, we consider the general case of product of traces instead of one trace. More precisely, we consider a family of independent $(r,s)$-Brownian motion $(B^{N,j}_{r,s})_{j\in J}$ indexed by a set $J$, which are copies of a same Brownian motions $B^N_{r,s}$ on $\GL_N$. We fix a family of time $\mx{T}=(t_j)_{j\in J}$, and we consider the family $B^N(\mx{T})=(B^{N,j}_{r,s}(t_j))_{j\in J}$. Let $Y^N_1,\ldots,Y^N_n$ be random variables of the form
%$$Y_i^N=\tr(P_1)\cdots \tr(P_k)$$ where each $P_l$ is a noncommutative polynomial in the elements of $B^N(\mx{T})$ and their adjoints. Theorem~\ref{main} says that the vector
%$$N(Y^N_i-\E[Y^N_i])_{1\leq i\leq n}$$
%is asymptotically Gaussian, with explicit covariances.

\begin{remark} Moreover, Theorem \ref{main} shows that the difference between any mixed moment in $X_1,\ldots,X_k$ and the corresponding mixed moment of the limit Gaussian distribution is $O(\frac{1}{N})$.  This implies that, in the language of \cite{MingoSpeicher2006}, the random matrices $B_{r,s}^N(\mx{T})$ possess a  second order distribution.  (Note we normalize the trace, while \cite{MingoSpeicher2006} uses the unnormalized trace, which accounts for the apparent discrepancy in normalizations.)  Since the random matrices $B_{r,s}^N(t)$ are unitarily invariant for each $t$, it then follows from \cite[Theorem 1]{MingoSpeicher2007} that the increments of $(B_{r,s}(t))_{t\ge 0}$ are {\em asymptotically free of second order}. \end{remark}

We can also explicitly describe the covariance of the fluctuations, and thus completely characterize them.  The full result is spelled out in Theorem \ref{t.main.cov}.  Here we state only one result of Corollary \ref{c.main.cov} (which already elucidates how the covariance extends from the unitary $(r,s)=(1,0)$ case).

\begin{theorem} \label{t.summ.cov} Let $(b_t)_{t\ge 0}$, $(c_t)_{t\ge 0}$, and $(d_t)_{t\ge 0}$ be freely independent free multiplicative $(r,s)$-Brownian motions in a tracial noncommutative probability space $(\mathscr{A},\tau)$ (for definitions, see Section \ref{s.fmbm}).  As in Theorem \ref{t.main.summ}, let $n=1$ and $\mx{T}=T$, and let $P_1,\ldots,P_k\in\C[X]$ be ordinary one-variable polynomials, with $X_1,\ldots,X_k$ denoting the fluctuations associated to $\tr P_1,\ldots,\tr P_n$.  Then their asymptotic Gaussian distribution has covariance $[\sigma_T(i,j)]_{1\le i,j\le k}$, where
\begin{equation} \label{e.ext.Levy} \sigma_T(i,j) = (r+s)\int_0^T \tau[P'(b_tc_{T-t})(Q'(b_td_{T-t}))^*]\,dt. \end{equation}
Here $P'$ is the derivative of $P$ relative to the unit circle:
\[ P'(z) = \lim_{h\to 0}\frac{f(ze^{ih})-f(z)}{h}. \]
\end{theorem}
Eq. \eqref{e.ext.Levy} generalizes \cite[Theorem 2.6]{LevyMaida2010}. As pointed out there, in this special case the covariances converge as $T\to\infty$ to the Sobolev $H_{1/2}$ inner-product of the involved polynomials, reproducing the main result of \cite{Diaconis2001} (as it must, since the heat kernel measure on $\U_N$ converges to the Haar measure in the large time limit).  For more general trace polynomial test functions, the covariance can always be described by such an integral, involving three freely independent free multiplicative Brownian motions in an input function built out of the {\em carr\'{e} du champ} intertwining operator determined by the $(r,s)$-Laplacian on $\GL_N$, cf.\ Section \ref{s.carreduchamp}.

Let us say a few words about the notation used in the formulation of Theorem~\ref{main}. In~\cite{Cebron2013} and in~\cite{Driver2013,Kemp2013b}, two different formalism were developed to handle general trace polynomial functions. Concretely, two different spaces were defined, namely $\C\{X_j,X_j^\ast\colon j\in J\}$ and $\PP(J)$; in Theorem \ref{t.main.summ}, $J=\{1,\ldots,k\}$. Each space leads to a functional calculus adapted to linear combinations of functions from $\GL_N^J$ to $\C$ of the form
$$(G_j)_{j\in J}\mapsto \tr (P_1(G_j,G_j^*:j\in J))\cdots \tr (P_k(G_j,G_j^*:j\in J)),$$
where $P_1,\ldots, P_k$ are noncommutative polynomials in elements of $(G_j)_{j\in J}$ and their adjoints.  In  Appendix \ref{appendix}, we investigate the relationship between these two spaces, demonstrating an explicit algebra isomorphism between a subspace of $\C\{X_j,X_j^\ast\colon j\in J\}$ and $\PP(J)$ for a given index set $J$. For notational convenience, most of the calculations throughout this paper (in particular in the proof of Theorem~\ref{main}) are expressed using the space $\PP(J)$, but all the results and proofs of this article can be transposed from $\PP(J)$ to $\C\{X_j,X_j^\ast\colon j\in J\}$ without major modifications.

The rest of the paper is organized as follows. In Section~\ref{background}, we give the definition of the $(r,s)$-Brownian motion as well as the definitions of $\PP(J)$, and we recall some results from \cite{Cebron2013,Kemp2013b}. Section~\ref{secmain} provides the statement of our main result Theorem~\ref{main}, an abstract description of the limit covariance matrix, and the proof of Theorem~\ref{main}. In Section~\ref{studycov}, we give an alternative description of the limit covariance, using three noncommutative processes in the framework of free probability, extending the results in~\cite{LevyMaida2010} from the unitary case to the general linear case, and beyond to all $(r,s)$. Finally, Appendix \ref{appendix} defines the equivalent abstract space $\C\{X_j,X_j^\ast\colon j\in J\}$ to encode trace polynomial functional calculus, and investigates the relationship between $\C\{X_j,X_j^\ast\colon j\in J\}$ and $\PP(J)$.

\section{Background}\label{background}

In this section, we briefly describe the basic definitions and tools used in this paper.  Section \ref{not} discusses  Brownian motions on $\GL_N$ (including the Brownian motion on $\U_N$ as a special case).  Section \ref{s.C{J}} addresses trace polynomials functions, and the two (equivalent) abstract intertwining spaces used to compute with them.  Section \ref{s.heat kernel} states the main structure theorem for the Laplacian that is used to prove the optimal asymptotic results herein.  Finally, Section \ref{s.fmbm} gives a brief primer on free multiplicative Brownian motion.  For greater detail on these topics, the reader is directed to the authors' previous papers \cite{Cebron2013,Driver2013,Kemp2013a,Kemp2013b}.

\subsection{Brownian motions on $\GL_N$}\label{not}

Fix $r,s>0$ throughout this discussion. Define the real inner product $\langle\cdot,\cdot\rangle_{r,s}^N$ on $\M_N$ by
$$\langle A,B\rangle_{r,s}^N=\frac{1}{2}\left(\frac{1}{s}+\frac{1}{r}\right)N\Re \Tr(AB^*)+\frac{1}{2}\left(\frac{1}{s}-\frac{1}{r}\right)N\Re \Tr(AB) .$$
As discussed in \cite{Kemp2013b}, this two-parameter family of metrics encompasses all real inner products on $\M_N = \mathrm{Lie}(\GL_N)$ that are invariant under conjugation by $\U_N$ in a strong sense that is natural in our context, and so we restrict our attention to diffusion processes adapted to these metrics.  An {\em $(r,s)$-Brownian motion} on $\GL_N$ is a diffusion process starting at the identity and with generator $\frac{1}{2}\Delta_{r,s}^N$, where $\Delta_{r,s}^N$ is the Laplace-Beltrami operator on $\GL_N$ for the left-invariant metric induced by $\langle\cdot,\cdot\rangle_{r,s}^N$. More concretely, if we fix a orthonormal basis $\beta_{r,s}^N$ of $\M_N$ for the inner-product $\langle\cdot,\cdot\rangle_{r,s}^N$, we have
\[ \Delta_{r,s}^N=\sum_{\xi\in\beta_{r,s}^N}\del_{\xi}^2, \]
where, for $\xi\in \M_N$, $\partial_\xi$ denotes the induced left-invariant vector field on $\GL_N$:
\[ (\partial_\xi f)(g)=\left.\frac{d}{dt}\right|_{t=0}f(g e^{t\xi}),\hspace{1cm} f\in C^{\infty}(\GL_N). \]

The $(r,s)$-Brownian motion $B^N_{r,s}(t)$ may also be seen as the solution of a stochastic differential equation, cf.\ \cite[Section 2.1]{Kemp2013b}.  Let $W_{r,s}^N(t)$ denote the diffusion on $\M_N$ determined by the $(r,s)$-metric; in other words, let $W_\xi(t)$ be i.i.d.\ standard $\R$-valued Brownian motions for $\xi\in\beta^N_{r,s}$, and take
\[ W_{r,s}^N(t) = \sum_{\xi\in\beta^N_{r,s}} W_\xi(t)\xi. \]
This can also be expressed in terms of standard $\mathrm{GUE}^N$-valued Brownian motions:
\begin{equation}\label{e.WrsSDE} W_{r,s}^N(t) = \sqrt{r}iX^N(t) + \sqrt{s}Y^N(t) \end{equation}
where $X^N(t)$ and $Y^N(t)$ are independent Hermitian matrices, with all i.i.d.\ upper triangular entries that are complex Brownian motions of variance $\frac{t}{N}$ above the main diagonal and real Brownian motions of variance $\frac{t}{N}$ on the main diagonal.  Then the $(r,s)$-Brownian motion on $\GL_N$ is the unique solution of the stochastic differential equation
\begin{equation} \label{e.BrsSDE} dB_{r,s}^N(t) = B_{r,s}^N(t)\,dW^N_{r,s}(t) - \frac12(r-s)B^N_{r,s}(t)\,dt, \end{equation}
cf.\ \cite[Equation (2.10)]{Kemp2013b}.

Fix an index set $J$; in this paper $J$ will usually be finite. For all $j\in J$, let $B^{j,N}_{r,s}=(B^{j,N}_{r,s}(t))_{t\geq 0}$ be independent $(r,s)$-Brownian motions\footnote{We could vary the parameters $r,s$ with $j\in J$ as well, with only trivial modifications to the following; at present, we do not see any advantage in doing so.} on $\GL_N$. Set $B^N=(B^{j,N}_{r,s})_{j\in J}$, which is the family of independent $(r,s)$-Brownian motions on $\GL_N$ indexed by $J$. The process $(B^N(t))_{t\geq 0}$ is therefore a diffusion process on $\GL_N^J$. More precisely, $(B^N(t))_{t\geq 0}$ is a Brownian motion on the Lie group $\GL_N^J$ for the metric $(\langle\cdot,\cdot\rangle_{r,s}^N)^{\tensor J}$. The reader is directed to \cite[Section 3.1]{Kemp2013b} for a discussion of the Laplace operators on $\GL_N^J$ for the metric $(\langle\cdot,\cdot\rangle_{r,s}^N)^{\tensor J}$.  The degenerate $(r,s)=(1,0)$ case gives the usual Laplacian on $\U_N^J$, while $(r,s)=(\frac12,\frac12)$ yields the canonical Laplacian on $\GL_N^J$ (induced by the scaled Hilbert-Schmidt inner product on $\M_N=\mathrm{Lie}(\GL_N)$).

%For all $\mx{T}=(t_j)_{j\in J}$, we denote by $B^N(\mx{T})$ the family $(B^{j,N}_{r,s}(t_j))_{j\in J}$. How to describe the joint law of the random vector $B^N(\mx{T})$ as the law of a process on $\GL_N^J$ stopped at a particular time? We fix now $\mx{T}=(t_j)_{j\in J}$ with $t_j\ge 0$. 
For each $j\in J$, let $\Delta^N_j$ denote the Laplacian on the $j$th factor of $\GL_N$ in $\GL_N^J$.  That is to say,
\[ \Delta^N_j= \sum_{\xi_j\in\beta^N_{r,s}} \del_{\xi_j}^2 \]
 where $\beta_{r,s}^N$ is an orthonormal basis of $\M_N$ for the inner product $\langle\cdot,\cdot\rangle_{r,s}^N$, and for all $\xi_j\in \beta_{r,s}^N$, $\del_{\xi_j}$ is the left-invariant vector field which acts only on the $j$th component of $\GL_N^J$. For $j\in J$, let $t_j\ge 0$, and set $\mx{T}=(t_j)_{j\in J}$. We consider the operator
\begin{equation} \label{e.T} \mx{T}\cdot\Delta^N = \sum_{j\in J} t_j\Delta^N_j. \end{equation}
%For $J$ finite and all $t>0$, it follows from standard theory that this operator is elliptic, essentially self-adjoint on $C_c^\infty(\GL_N^J)$, and non-positive.  We may therefore use the spectral theorem to define the bounded operator
%\[ e^{\frac12\mx{T}\cdot\Delta^N} = \prod_{j\in J}e^{\frac{t_j}{2}\Delta^N_j}. \]
\begin{definition}For $J$ finite, denote by  $(B^N(t\mx{T}))_{t\geq 0}$ the diffusion process on $\GL_N^J$ with generator $\mx{T}\cdot\Delta^N$.\label{genbnt}
\end{definition}
\noindent We could write down a stochastic differential equation for $B^N(t\mx{T})$ similar to \eqref{e.BrsSDE}; for our purposes, we only need the fact that it is a diffusion process.

A common computational tool used throughout \cite{Cebron2013,Driver2013,Kemp2013a,Kemp2013b} is the collection of so-called ``magic formulas."  In the present context, the form needed is as follows; cf.\ \cite[Equations~(2.7)~and~(3.6)]{Kemp2013b}.

\begin{proposition}Let $\beta^N_{r,s}$ be any orthonormal basis of $\M_N$ for the inner product $\langle\cdot,\cdot\rangle_{r,s}^N$. Then, for any $A,B\in \M_N$, we have\label{magicformula}
\[  \sum_{\xi\in \beta^N_{r,s}}\tr(\xi A)\tr(\xi B)=\sum_{\xi\in \beta^N_{r,s}}\tr(\xi^* A)\tr(\xi^* B)=\frac{1}{N^2}(s-r)\tr(AB), \]
and
\[ \sum_{\xi\in \beta^N_{r,s}}\tr(\xi^* A)\tr(\xi B)=\frac{1}{N^2}(s+r)\tr(AB). \]
\end{proposition}

\subsection{The Space $\PP(J)$\label{s.C{J}}}

Let $J$ be an index set as above, and let $\mx{A}=(A_j)_{j\in J}$ be a collection of matrices in $\M_N$. A {\em trace polynomial} function on $\M_N$ is a linear combination of functions of the form
\[ \mx{A}\mapsto P_0(\mx{A})\tr(P_1(\mx{A}))\tr(P_2(\mx{A}))\cdots \tr(P_m(\mx{A})) \]
for some finite $m$, where $P_1,\ldots,P_m\in \C$ are noncommutative polynomials in $J\times\{1,\ast\}$ variables (i.e.\ the polynomials may depend explicitly on $A_j$ and $A_j^\ast$ for all $j\in J$).  Such functions arise naturally in our context: applying the operator $\mx{T}\cdot\Delta^N$ to the smooth function $\mx{A}\mapsto Q(\mx{A})$ for any noncommutative polynomial $Q$ generally results in a trace polynomial function.
%Let $P$ a noncommutative polynomial in $\C\langle J\rangle \equiv \C\langle X_j,X_j^\ast\colon j\in J\rangle$. When applying the operator $\mx{T}\cdot\Delta^N$ on the smooth function
%$(G_1,\cdots, G_j)\mapsto P(G_1,\ldots,G_j),$
%it makes appear a linear combination of functions of the form
%$$(G_1,\cdots, G_j)\mapsto P_0(G_1,\ldots,G_j)\tr (P_1(G_1,\ldots,G_j))\cdots \tr (P_k(G_1,\ldots,G_j)),$$
%where $P_0,\ldots, P_k\in \C\langle J\rangle$. 
The vector space of trace polynomial functions is closed under the action of $\mx{T}\cdot\Delta^N$, and this is a motivation for defining abstract spaces which encodes them. In~\cite{Cebron2013} and in~\cite{Driver2013,Kemp2013b}, two different spaces are defined, namely $\C\{X_j,X_j^\ast\colon j\in J\}$ and $\PP(J)$. In this section, we present the space $\PP(J)$, and the relation between $\C\{X_j,X_j^\ast\colon j\in J\}$ and $\PP(J)$ can be found in Appendix~\ref{appendix}.

First we give the definition of the space $\PP(J)$.  Let $\EX(J) = \bigcup_{n\geq 1} (J\times\{ 1,\ast\})^n$ be the set of all words whose letters are pairs of the form $(j,1)$, or $(j,\ast)$ for some $j\in J$.  Let $\mx{v}_J = \{v_\ex\colon \ex\in\EX(J)\}$ be commuting variables indexed by all such words, and let
\[ \PP(J) \equiv \C[\mx{v}_J] \]
be the algebra of \emph{commutative} polynomials in these variables.  That is, $\PP(J)$ is the vector space with basis $1$ together with all monomials
\begin{equation*} v_{\ex^{(1)}}\cdots v_{\ex^{(k)}}, \quad k\in\N, \quad \ex^{(1)},\ldots,\ex^{(k)}\in \EX(J), \end{equation*}
and the (commutative) product on $\PP(J)$ is the standard polynomial product.
\begin{remark}One can think of $\PP(J)$ as a particular framework of noncommutative functional calculus. Instead of considering tensor products of $\C\langle X_j,X^*_j:j\in J \rangle$ as usually in free probability, we consider symmetric tensor product of $\C\langle X_j,X^*_j:j\in J \rangle$, or equivalently, commutative polynomials in words in $(j,1)$, or $(j,\ast)$. It turns out that the commutativity of the product in $\PP(J)$ is very convenient in the forthcoming computations.\end{remark}
We present now the following notions of degree, evaluation and conjugation (see~\cite{Kemp2013b} for a detailed presentation):
\begin{itemize}\item
In \cite[Definition 3.2]{Kemp2013b}, the notion of {\em degrees} of elements in the space $ \PP(J) $ is defined:
$$\deg( v_{\ex^{(1)}}\cdots v_{\ex^{(k)}}) = |\ex^{(0)}|+\cdots+|\ex^{(k)}|,$$
% \\
%\deg_{\C\{J\}}(M_0\tr(M_1)\cdots\tr(M_k)) &= \deg_{\C\langle J\rangle} M_0+\cdots+\deg_{\C\langle J\rangle} M_k, \end{align*}
where $|\ex|$ is the length of the string $\ex$.

\item Let $(\mathscr{A},\tau)$ be a noncommutative probability space (cf.\ Section \ref{s.fmbm}). For each $\ex = ((j_1,\ex_1),\ldots,(j_n,\ex_n))$, there is an evaluation function
$[v_\ex]_{(\mathscr{A},\tau)}\colon \mathscr{A}^J\to \C$
given, for each $\mx{a}=(a_j)_{j\in J}\in\A^J$, by
%the function $\mx{V}^J_{(\mathscr{A},\tau)}(\mx{a})\colon\EX(J)\to\C$ is given by
\[ [v_\ex]_{(\mathscr{A},\tau)}(\mx{a}) = \tau(a_{j_1}^{\ex_1}\cdots a_{j_n}^{\ex_n}).\]Note, the $\ast$ is no longer a formal symbol here: $a_j^*$ means the adjoint of $a_j$ in $\mathscr{A}$. More generally, for all $ P\in\PP(J)=\C[\mx{v}_J]$, we define $[P]_{(\mathscr{A},\tau)}: \A^J\to \C$ by saying that, for all $\mx{a}\in\A^J$, the maps $P\mapsto [P]_{(\mathscr{A},\tau)}(\mx{a})$ are algebra homomorphisms from $\PP(J)=\C[\mx{v}_J]$ to $\C$. Let us emphasize that it implies the following commutativity between the evaluation and the product: for all polynomials $P,Q\in\PP(J)$ and $\mx{a}\in\A^J$, we have
$$[PQ]_{(\mathscr{A},\tau)}(\mx{a})=[P]_{(\mathscr{A},\tau)}(\mx{a})\cdot [Q]_{(\mathscr{A},\tau)}(\mx{a}).$$ 
%
%may be thought of as a function on $\C^{\EX(J)}$; indeed, these are the only functions on the abstract space $\C^{\EX(J)}$ that can be made sense of without additional analytic or metric structure.  We may therefore compose $\mx{V}^J_{(\mathscr{A},\tau)}$ with $P$, to give a $\C$-valued function $[P]_{(\mathscr{A},\tau)}=P\circ\mx{V}^J_{(\mathscr{A},\tau)}$ on $\mathscr{A}^J$. Note that more explicitly, for all 
In the particular case where $(\mathscr{A},\tau)=(\GL_N,\tr)$, we will simply denote the map $[P]_{(\mathscr{A},\tau)}$ by $[P]_N$.
We finally remark that if $\mx{a}=(a_j)_{j\in J}$ with $a_j=1_{\mathscr{A}}$ for all $j\in J$, then $[P]_{(\mathscr{A},\tau)}(\mx{a})$ does not depend on the space $(\mathscr{A},\tau)$, and we will simply denote it by
$$P(\mx{1})\equiv[P]_{(\mathscr{A},\tau)}(\mx{a}). $$
\item There is a natural notion of conjugation on $\PP(J)$: $P^\ast$ is the result of taking complex conjugates of all coefficients, and reversing $1\leftrightarrow\ast$ in all indices.  In terms of evaluation as a trace polynomial function, we have $[P^\ast]_N = \overline{[P]_N}$, cf.\ \cite[Lemma 3.17]{Kemp2013b}.
\end{itemize}

\subsection{Computation of the heat Kernel \label{s.heat kernel}}
We are now able to see how the Laplacian acts on the space of trace polynomial functions (i.e.\ functions on $\M_N$ given by evaluations $[P]_N$ of $P\in\PP(J)$).
%\begin{theorem}Let $j\in J$. There are collections $\{Q^j_\epsilon:\epsilon \in \EX\}$ and $\{R^j_{\epsilon,\delta}:\epsilon, \delta \in \EX\}$ such that:
%\begin{enumerate}
%\item 
%\item
%\end{enumerate}
%\end{theorem}
%We define the following operators on $\PP$: $$\D^{\mx{T}}=\frac{1}{2}\sum_{\substack{j\in J\\ \epsilon \in \EX}}t_jQ_{\epsilon}^j\frac{\partial}{\partial v_\epsilon}, \hspace{2cm}\L^{\mx{T}}=\frac{1}{2}\sum_{\substack{j\in J\\ \epsilon,\delta \in \EX}}t_jR_{\epsilon,\delta }^j\frac{\partial^2}{\partial v_\epsilon \partial v_\delta} ,$$
%and we define also the symmetric bilinear form on $\PP\times \PP$
%$$\Gamma^{\mx{T}}(P,Q )=\frac{1}{2}\left(\L^{\mx{T}}(PQ)-\L^{\mx{T}}(P)Q-P\L^{\mx{T}}(Q)\right)=\frac{1}{2}\sum_{\substack{j\in J\\ \epsilon,\delta \in \EX}}t_jR_{\epsilon,\delta }^j\frac{\partial}{\partial v_\epsilon}(P)\frac{\partial}{ \partial v_\delta}(Q). $$
\begin{theorem}{\cite[Theorems 3.8 and 3.9]{Kemp2013a}\label{deltaonp}} Let $\mx{T}$ be as in \eqref{e.T} above. There exist two linear operators $\D^{\mx{T}}$ and $\L^{\mx{T}}$ on $\PP(J)$, independent from $N$, such that:
\begin{enumerate}
\item $\D^{\mx{T}}$ is a first-order operator, i.e. for all $P,Q\in \PP(J)$, $\D^{\mx{T}}(PQ)=\D^{\mx{T}}(P)Q+P\D^{\mx{T}}(Q)$;
\item $\L^{\mx{T}}$ is a second-order operator, i.e. for all $P,Q,R\in \PP(J)$, $$\L^{\mx{T}}(PQR)=\L^{\mx{T}}(PQ)R+P\L^{\mx{T}}(QR)+\L^{\mx{T}}(PR)Q-\L^{\mx{T}}(P)QR-P\L^{\mx{T}}(Q)R-PQ\L^{\mx{T}}(R);$$
\item For all $P\in \PP(J)$,  $(\mx{T}\cdot\Delta^N)([P]_N)=\left[(\D^{\mx{T}}+\frac{1}{N^2}\L^{\mx{T}})(P)\right]_N$.
\end{enumerate}
\end{theorem}
\begin{remark}In \cite[Section 3.3]{Kemp2013a}, there is an inductive definition of $\D^{\mx{T}}$ and $\L^{\mx{T}}$ which are denoted similarly. In \cite[Sections 4.1 and 4.2]{Cebron2013}, there is an explicit definition of $\D^{\mx{T}}$ in the simple cases of $J=\{1\}$ and $(r,s)=(1,0)$ or $(r,s)=(\frac12,\frac12)$, which corresponds respectively to $\Delta_U$ and $\Delta_{GL}$, and of $\L^{\mx{T}}$ in the same simple cases, which corresponds respectively to $\tilde{\Delta}_U$ and $\tilde{\Delta}_{GL}$. Since we don't need any more details about $\D^{\mx{T}}$ and $\L^{\mx{T}}$, we refer to \cite{Cebron2013,Driver2013,Kemp2013a,Kemp2013b} for further informations about those operators.
\end{remark}
Using Definition~\ref{genbnt}, we deduce the following result from Theorem~\ref{deltaonp}.
\begin{corollary} \label{c.deltaonp} Let $B^N=(B^{j,N}_{r,s})_{j\in J}$ be a collection of independent $(r,s)$-Brownian motions on $\GL_N$. Let $P\in \PP(J)$. We have\label{exbt}
$$\E\left([P]_N(B^N(\mx{T}))\right)=\left[e^{\D^{\mx{T}}+\frac{1}{N^2}\L^{\mx{T}}}(P)\right](\mx{1}). $$
\end{corollary}
\noindent This is merely the statement, in the present language, of the fact that the expectation of any function of a diffusion can be computed by applying the associated heat semigroup to the function and evaluating at the starting point.

\subsection{Free Multiplicative Brownian Motion\label{s.fmbm}}

Here we give a very brief description of free stochastic processes, and free probability in general.  For a complete introduction to the tools of free probability, the best source is the \cite{Nica2006}.  For brief summaries of central ideas and tools from free stochastic calculus, the reader is directed to \cite[Section 1.2-1.3]{CollinsKemp2014},  \cite[Section 2.4-2.5]{Kemp2013a}, \cite[Section 2.7]{Kemp2013b}, and \cite[Section 1.1-1.2]{KNPS2012}.

A {\em noncommutative probability space} is a pair $(\mathscr{A},\tau)$ where $\mathscr{A}$ is a unital algebra of operators on a (complex) Hilbert space, and $\tau$ is a (usually tracial) state on $\mathscr{A}$: a linear functional $\tau\colon\mathscr{A}\to\C$ such that $\tau(1)=1$ and $\tau(ab)=\tau(ba)$.  Typical examples are $\mathscr{A}=\M_N, \tau=\tr$ (deterministic matrices), or $\mathscr{A}=\M_N\tensor L^{\infty-}(\P), \tau = \tr\tensor \E_\P$ (random matrices with entries having moments of all orders).  In infinite-dimensional cases, it is typical to add other topological and continuity properties to the pair $(\mathscr{A},\tau)$ that we will not elaborate on presently.  Elements of the algebra $\mathscr{A}$ are generally called random variables.  In any noncommutative probability space, one can speak of the {\em noncommutative distribution} of a collection of random variables $a_1,\ldots,a_n\in\mathscr{A}$: it is simply the collection of all mixed moments in $a_1,\ldots,a_n,a_1^\ast,\ldots,a_n^\ast$; that is the collection $\tau[P(a_j,a_j^\ast)_{1\le j\le n}]$ for all noncommutative polynomials $P$ in $2n$ variables.  We then speak of convergence in noncommutative distribution: if $(\mathscr{A}_N,\tau_N)$ are noncommutative probability spaces, a sequence $(a_1^N,\ldots,a_n^N)\in\mathscr{A}_N^n$ converges in distribution to $(a_1,\ldots,a_n)\in\mathscr{A}^n$ if
\[ \tau[P(a^N_j,(a^N_j)^\ast)_{1\le j\le n}] \to \tau[P(a_j,a_j^\ast)_{1\le j\le n}] \; \text{as} \;N\to\infty, \quad \text{for each }P. \]

{\em Free independence} (sometimes just called freeness) is an independence notion in any noncommutative probability space.  Two random variables $a,b\in\mathscr{A}$ are freely independent if, given any $n\in\N$ and any noncommutative polynomials $P_1,\ldots,P_n,Q_1,\ldots,Q_n$ each in two variables which are such that $\tau(P_j(a,a^\ast))=\tau(Q_j(b,b^\ast))=0$ for each $j$, it follows that $\tau(P_1(a,a^\ast)Q_1(b,b^\ast)\cdots P_n(a,a^\ast)Q_n(b,b^\ast)) =0$.  This gives an algorithm for factoring moments: it implies that $\tau(a^nb^m)=\tau(a^n)\tau(b^m)$ for any $n,m\in\N$, just as holds for classically independent random variables, but it also includes higher-order noncommutative polynomial factorizations; for example $\tau(abab) = \tau(a^2)\tau(b)^2 + \tau(a)^2\tau(b^2)-\tau(a)^2\tau(b)^2$.  One finds freely independent random variables typically only in infinite-dimensional noncommutative probability spaces, although random matrices often exhibit asymptotic freeness (i.e.\ they converge in noncommutative distribution to free objects).

In \cite{Voiculescu1991}, Voiculescu showed that there exists a noncommutative probability space (any free group factor, for example) that possesses limits $x(t),y(t)$ of the matrix-valued diffusion processes $X^N(t),Y^N(t)$ of \eqref{e.WrsSDE} that are freely independent.  Note that this convergence is not just for each $t$ separately, but for the whole process: convergence of the finite-dimensional noncommutative distributions.  The one-parameter families $x(t),y(t)$ are known as (free copies of) {\em additive free Brownian motion}.  We refer to them as free stochastic processes, although they are deterministic in the classical sense.

There is an analogous theory of stochastic differential equations in free probability, cf.\ \cite{Biane1998a,BianeSpeicher2001}.  One may construct stochastic integrals with respect to free additive Brownian motion, precisely mirroring the classical construction.  In sufficiently rich noncommutative probability spaces (such as the one Voiculescu dealt with in \cite{Voiculescu1991}), free It\^o stochastic differential equations of the usual form
\[ dm(t) = \mu(t,m(t))\,dt + \sigma(t,m(t))\,dx(t), \]
have unique long-time solutions with a given initial condition, assuming standard continuity and growth conditions on the drift and diffusion coefficient functions $\mu,\sigma$.  In particular, letting $w_{r,s}(t) = \sqrt{r}ix(t)+\sqrt{s}y(t)$ (mirroring \eqref{e.WrsSDE}), the free stochastic differential equation analogous to \eqref{e.BrsSDE},
\[ db_{r,s}(t) = b_{r,s}(t)\,dw_{r,s}(t) - \frac12(r-s)b_{r,s}(t)\,dt, \quad b_{r,s}(0)=1, \]
has a unique solution which we call {\em free multiplicative $(r,s)$-Brownian motion}.  In the special case $(r,s)=(1,0)$, the resulting process takes values in unitary operators and is known as {\em free unitary Brownian motion}; when $(r,s)=(\frac12,\frac12)$, it is known as (standard) {\em free multiplicative Brownian motion}.  Both were introduced in \cite{Biane1997a}, where it was proven that the process $(B^N_{1,0}(t))_{t\ge 0}$ converges to the process $(b_{1,0}(t))_{t\ge 0}$.  The main theorem of \cite{Kemp2013b} is the corresponding convergence result for the general processes $(B^N_{r,s}(t))_{t\ge 0}$ to $(b_{r,s}(t))_{t\ge 0}$.

\section{Gaussian Fluctuations}\label{secmain}

In this section, we prove our main Theorem \ref{main}, which is summarized in the slightly weaker form of Theorem \ref{t.main.summ} in the Introduction.  To begin, in Section \ref{s.carreduchamp} we set the stage with the main tool involved in the computation: the {\em carr\'{e} du champ} operator associated to the Laplacian on $\GL_N^J$, which measures the defect of this second order operator from satisfying the product rule.  Section \ref{s.maintheorem} then gives the statement of our Main Theorem \ref{main} and associated results that together yield the Gaussian fluctuations of the $\GL_N$ Brownian motions.  Finally, Section \ref{s.main.proof} is devoted to the proof of Theorem \ref{main}.

\subsection{The carr\'{e} du champ of $\mx{T}\cdot\Delta^N$}\label{s.carreduchamp}
 We define the  {\em carr\'{e} du champ} operator of $\mx{T}\cdot\Delta^N$ for all twice continuously differentiable $f,g:\GL_N^J \to \C$ by
$$\Gamma^{\mx{T}}_N(f,g)=\frac{1}{2}\left((\mx{T}\cdot\Delta^N)(fg)-(\mx{T}\cdot\Delta^N)(f)g-f(\mx{T}\cdot\Delta^N)(g)\right),$$
or equivalently by
$$\Gamma^{\mx{T}}_N(f,g)=\frac{1}{2}\sum_{\xi\in \beta_{r,s}^N, j\in J}t_j \cdot (\partial_{\xi_j} f )(\partial_{\xi_j} g).$$
As with the operator $\mx{T}\cdot\Delta^N$ in Theorem~\ref{deltaonp}, the operator $\Gamma^{\mx{T}}_N$ is the push forward of an operator on $\PP(J)$ as follows. Let us define the symmetric bilinear form on $\PP(J)\times \PP(J)$ by
$$\Gamma^{\mx{T}}(P,Q )=\frac{1}{2}\left(\L^{\mx{T}}(PQ)-\L^{\mx{T}}(P)Q-P\L^{\mx{T}}(Q)\right)
%=\frac{1}{2}\sum_{\substack{j\in J\\ \epsilon,\delta \in \EX}}t_jR_{\epsilon,\delta }^j\frac{\partial}{\partial v_\epsilon}(P)\frac{\partial}{ \partial v_\delta}(Q)
.$$
\begin{proposition}\label{p.carre}
For all $P,Q\in \PP(J)$, we have $N^2\Gamma^{\mx{T}}_N([P]_N,[Q]_N )=\left[\Gamma^{\mx{T}}(P,Q)\right]_N.$
\end{proposition}
\begin{proof}Let us denote by $\D^{\mx{T}}_N$ the operator $\D^{\mx{T}}+\frac{1}{N^2}\L^{\mx{T}}$. We have $\D^{\mx{T}}(PQ)-\D^{\mx{T}}(P)Q-P\D^{\mx{T}}(Q)=0$. As a consequence,
$\Gamma^{\mx{T}}(P,Q )=\frac{N^2}{2}\left(\D_N^{\mx{T}}(PQ)-\D_N^{\mx{T}}(P)Q-P\D_N^{\mx{T}}(Q)\right)$. Using $(\mx{T}\cdot\Delta^N)([P]_N)=\left[\D^{\mx{T}}_N(P)\right]_N$, we obtained that
$$ \left[\Gamma^{\mx{T}}(P,Q)\right]_N=\frac{N^2}{2}\left((\mx{T}\cdot\Delta^N)([PQ]_N)-(\mx{T}\cdot\Delta^N)([P]_N)\cdot[Q]_N-[P]_N\cdot (\mx{T}\cdot\Delta^N)([Q]_N)\right),$$
which is the {\em carr\'{e} du champ} of $\mx{T}\cdot\Delta^N$, as wanted.
\end{proof}
Since $\L^{\mx{T}}$ is a second-order differential operator, we have the following.
\begin{lemma}For all $P,Q,R\in \PP(J)$,
$$\Gamma^{\mx{T}}(PQ,R)=\Gamma^{\mx{T}}(P,R)\cdot Q+P\cdot \Gamma^{\mx{T}}(Q,R) .$$
Additionally, for all $P_1,\ldots,P_k\in \PP(J)$, \label{landg}
$$\L^{\mx{T}}(P_1\cdots P_k)=\sum_{i=1}^k P_1\cdots \widehat{P_i}\cdots P_k \L^{\mx{T}}(P_i)+2\sum_{1\leq i<j\leq k} P_1\cdots \widehat{P_i}\cdots \widehat{P}_j\cdots P_k \Gamma^{\mx{T}}(P_i,P_j),
$$
where the hats mean that we omit the correspond terms in the product.
\end{lemma}

\begin{proof}Using the second-order property of $\L^{\mx{T}}$, we compute
\begin{align*}2\Gamma^{\mx{T}}(PQ,R)&=\L^{\mx{T}}(PQR)-\L^{\mx{T}}(PQ)R-PQ\L^{\mx{T}}(R)\\
&=\L^{\mx{T}}(PR)Q-\L^{\mx{T}}(P)QR-PQ\L^{\mx{T}}(R)\\
&\qquad\qquad +P\L^{\mx{T}}(QR)-P\L^{\mx{T}}(Q)R-PQ\L^{\mx{T}}(R)\\
&=2\Gamma^{\mx{T}}(P,R)\cdot Q+2P\cdot \Gamma^{\mx{T}}(Q,R).
\end{align*}
By a direct induction, we deduce that
\begin{align*}\L^{\mx{T}}(P_1\cdots P_k)=&\L^{\mx{T}}(P_1\cdots P_{k-1})P_k+P_1\cdots P_{k-1}\L^{\mx{T}}(P_k)+2\Gamma^{\mx{T}}(P_1\cdots P_{k-1},P_k)\\
=&\L^{\mx{T}}(P_1\cdots P_{k-1})P_k+P_1\cdots P_{k-1}\L^{\mx{T}}(P_k)+2\sum_{1\leq i\leq k} P_1\cdots \widehat{P_i}\cdots  P_{k-1} \Gamma^{\mx{T}}(P_i,P_k)\\
=&\cdots\\
=&\sum_{i=1}^k P_1\cdots \widehat{P_i}\cdots P_k \L^{\mx{T}}(P_i)+2\sum_{1\leq i<j\leq k} P_1\cdots \widehat{P_i}\cdots \widehat{P}_j\cdots P_k \Gamma^{\mx{T}}(P_i,P_j).\qedhere
\end{align*}
\end{proof}

\subsection{Main theorem\label{s.maintheorem}}
For all $P,Q\in\PP(J)$, define \[X_P^N \equiv N\left([P]_N(B^N(\mx{T}))-\E[[P]_N(B^N(\mx{T}))]\right)\] and
\[ \sigma_{\mx{T}}(P,Q) \equiv 2\int_0^1 \left[e^{t\D^{\mx{T}}}\left(\Gamma^{\mx{T}}(e^{(1-t)\D^{\mx{T}}}P,e^{(1-t)\D^{\mx{T}}}Q)\right)\right](\mx{1})\,dt. \]
Note that $P\in\PP(J)$, and the finite-dimensional subspace of elements with degree lower than or equal to the degree of $P$ is invariant under $\D^{\mx{T}}$ (cf.\ \cite[Corollary 3.10]{Kemp2013b}).  Hence, $e^{(1-t)\D^{\mx{T}}}$ makes sense in this context.  The same argument applied twice more shows that the integrand makes sense, and the finite-dimensionality of all involved polynomials yields continuity, so the integral is perfectly well-defined.

The following theorem says that the quantities of the form $\E(X^N_{P_1}\cdots X^N_{P_k})$ satisfy a Wick formula in large dimension, with covariances given by $ \sigma_{\mx{T}}$. Let us denote by $\mathcal{P}_2(k)$ the set of (unordered) pairings of $\{1,\ldots,k\}$.
\begin{theorem} \label{main}
For any $P_1,\ldots,P_k\in\PP(J)$, we have
\[ \E(X^N_{P_1}\cdots X^N_{P_k}) = \sum_{\pi\in\mathcal{P}_2(k)}\prod_{\{i,j\}\in\pi} \sigma_{\mx{T}}(P_i,P_j) + O\left(\frac{1}{N}\right). \]
\end{theorem}
Theorem \ref{main} is proved in the next section.  We will first reformulate this result as convergence towards a Gaussian field.
\begin{lemma} \label{l.Wick} There exists a complex Gaussian Hilbert space $K$ (cf.\ \cite{Janson1997}) with some specified random variables $(\xi_P)_{P\in \PP}\in K$ such that $P\mapsto \xi_P$ is linear, $\E(\xi_P\xi_Q)=\sigma_{\mx{T}}(P,Q)$ and $\overline{\xi_P}=\xi_{P^*}$.
\end{lemma}

\begin{proof}Firstly, the map $ \sigma_{\mx{T}}$ is symmetric, non-negative and bilinear on the subspace $\PP_{sa}$ of self-adjoint elements of $\PP(J)$, and therefore there exists a real Gaussian Hilbert space $H$ and a linear map $P\mapsto \xi_P$ from $\PP_{sa}$ to $H$ such that $\E(\xi_P\xi_Q)= \sigma_{\mx{T}}(P,Q)$. Let $K=H_{\mathbb{C}}$, the complexification of $H$. For all $P\in \PP$, we set $\xi_P=\xi_{(P+P^*)/2}+i\xi_{(P-P^*)/2i}$ which is linear in $P$. By bilinearity of $ \sigma_{\mx{T}}$, $\E(\xi_P\xi_Q)= \sigma_{\mx{T}}(P,Q)$. Finally,
\[ \overline{\xi_P}=\xi_{(P+P^*)/2}-i\xi_{(P-P^*)/2i}=\xi_{(P^*+P)/2}+i\xi_{(P^*-P)/2i}=\xi_{P^*}.\qedhere \]
\end{proof}

\begin{corollary} \label{c.main} As $N\to\infty$, $(X^N_P)_{P\in \PP(J)}$ converges to $(\xi_P)_{P \in \PP(J)}$ in finite-dimensional distribution: for all $P_1,\ldots,P_k\in \PP(J)$,
$$(X^N_{P_1},\ldots , X^N_{P_k})\xrightarrow[N\rightarrow\infty]{(d)} (\xi_{P_1},\ldots , \xi_{P_k}) .$$
Otherwise stated, in the dual space $\PP(J)^*$ endowed with the topology of pointwise convergence, the random linear map $X^N:P\mapsto X^N_P$ converge to the random linear map $\xi:P\mapsto \xi_P$ in distribution:
$$X^N \xrightarrow[N\rightarrow\infty]{(d)} \xi .$$
\end{corollary}
\noindent Note that, for $P$ and $Q$ in $\PP(J)$, the asymptotic covariance of $X^N_P$ and $X^N_Q$, or equivalently the covariance of $\xi_P$ and $\xi_Q$, is $\E(\xi_P\overline{\xi_Q})=\E(\xi_P\xi_{Q^*})=\sigma_{\mx{T}}(P,Q^*)$, which is different from $\sigma_{\mx{T}}(P,Q)$.

\begin{proof}Let $k\in \N$ and $P_1,\ldots,P_k\in\PP(J)$. Because the vector $(\xi_{P_1},\ldots,\xi_{P_k})$ is Gaussian, it suffices to prove the convergence of the $*$-moments of $(X^N_{P_1},\ldots, X^N_{P_k})$ to those of $(\xi_{P_1},\ldots,\xi_{P_k})$.
Let $1\leq i_1,\ldots, i_n,j_1,\ldots, j_m \leq k$. We want to prove that
$$\E(X^N_{i_1}\cdots X^N_{i_n}\overline{X^N_{j_1}}\cdots \overline{X^N_{j_n}})\xrightarrow[N\rightarrow\infty]{} \E(\xi_{P_{i_1}}\cdots \xi_{P_{i_n}}\overline{\xi_{P_{j_1}}}\cdots \overline{\xi_{P_{j_n}}}).$$
We have
 \begin{multline*} \E(X^N_{P_{i_1}}\cdots X^N_{P_{i_n}}\overline{X^N_{P_{j_1}}}\cdots \overline{X^N_{P_{j_m}}})= \E(X^N_{P_{i_1}}\cdots X^N_{P_{i_n}}X^N_{P_{j_1}^*}\cdots X^N_{P_{j_m}^*})\\
 \xrightarrow[N\rightarrow\infty]{} \E(\xi_{P_{i_1}}\cdots \xi_{P_{i_n}}\xi_{P_{j_1}^*}\cdots \xi_{P_{j_n}^*})
=\E(\xi_{P_{i_1}}\cdots \xi_{P_{i_n}}\overline{\xi_{P_{j_1}}}\cdots \overline{\xi_{P_{j_n}}}).
 \end{multline*}
 The equivalent formulation of the convergence in distribution follows because $\PP(J)$ is a countable-dimensional metric space.
\end{proof}

\subsection{Proof of Theorem~\ref{main}} \label{s.main.proof}

Observing that $(P_1,\ldots, P_k)\mapsto \E(X^N_{P_1}\cdots X^N_{P_k}) $ and $(P_1,\ldots, P_k)\mapsto\sum_{\pi\in\mathcal{P}_2(k)}\prod_{(i,j)\in\pi}  \sigma_{\mx{T}}(P_i,P_j) $ are symmetric multilinear forms on $\PP(J)$, it suffices by polarization to verify the asymptotic when $P_1=\cdots=P_k=P$ (cf.\ \cite[Appendix D]{Janson1997}). In this case, set $Q_N=P-\E[P(B^N(\mx{T}))]$.  (Note that $Q_N$ is an element of the abstract space $\PP(J)$; it should not be confused with the notation $[Q]_N$ for evaluation as a trace polynomial function on $\M_N$.)  We want to prove that
\[ N^k \E([Q_N^k]_N(B^N(\mx{T}))) = \sum_{\pi\in\mathcal{P}_2(k)}\prod_{(i,j)\in\pi}  \sigma_{\mx{T}}(P,P) + O\left(\frac{1}{N}\right). \]
To begin, we remark that 
\[ \E(Q_N^k(B^N(\mx{T})))= \left[e^{\D^{\mx{T}}+\frac{1}{N^2}\L^{\mx{T}}}(Q_N^k)\right](\mx{1}), \]
thanks to Corollary~\ref{exbt}. The proof will consist in identifying the leading term in the development of $e^{\D^{\mx{T}}+\frac{1}{N^2}\L^{\mx{T}}}$ with respect to $N$.

\paragraph*{Appropriate norms.} In order to control the negligible terms in the development, we will work on finite dimensional spaces. Let $d\in \N$ be the degree of $Q_N$ (which is independent of $N$). The subalgebra $\PP_{kd}$ of elements of $\PP(J)$ whose degrees are $\le kd$ is finite dimensional and we endow it with some fixed unital algebra norm $\|\cdot\|_{(kd)}$. Let us denote by $ \vertiii{\,\cdot\,}_{(kd)}$ the induced operator norm on the finite dimensional algebra $\End(\PP_{kd})$, and by $\vertiii{\,\cdot\,}_{(d,d')}$ the induced norm of bilinear maps from $\PP_{d}\times \PP_{d'}$ to $\PP_{d+d'}$ when $d+d'\leq kd$ (in the following development, we will often omit the indices $(kd)$ or $(d,d')$). Throughout this proof, we will denote by $D$, $L$ and $\Gamma$ the operators $\D^{\mx{T}}$, $\L^{\mx{T}}$ and $\Gamma^{\mx{T}}$ restricted to the finite dimensional algebra $\PP_{kd}$. Let us denote by respectively $O(1/N^2)$, $\textsc{O}(1/N^2)$ and $\mx{O}(1/N^2)$ the class of elements $A(N)$ in respectively $\C$, $\PP_{kd}$ and $\End(\PP_{kd})$ such that $|A(N)|$ (resp. $\|A(N)\|_{(kd)}$ and $\vertiii{A(N)}_{(kd)}$) is $\leq C/N^2$ for some constant $C$.
% The differentiability of the exponential map leads to
% $e^{D+\frac{1}{N^2}L}=e^D+O(1/N^2)$. More precisely, w
 We have the following result.
% Since the map $A\mapsto [A(P)](\mx{1})$ is linear, it is therefore bounded and we have $\left[e^{D+\frac{1}{N^2}}(P)\right](\mx{1})=[e^D(P)](\mx{1})+O(1/N^2)$ and $Q_N=P-e^DP(\mx{1})+O(1/N^2)$.

%Let $Q_N(P_j) = P_j - \E(P_j(B_t^N))$. Then
%$\E(X_1^N\cdots X_k^N) = N^k \left(e^{D+\frac{1}{N^2}L}(Q_N(P_j)\cdots Q_N(P_j))\right)(\mx{1}).$ Observing that $P_1,\ldots, P_k\mapsto N^k \left(e^{D+\frac{1}{N^2}L}(Q^N(P_j)\cdots Q_N(P_j))\right)(\mx{1})$ and $P_1,\ldots, P_k\mapsto\sum_{\pi\in\mathcal{P}_2(k)}\prod_{(i,j)\in\pi}  \sigma_{\mx{T}}(P_i,P_j) $ are symmetric multilinear forms on $\PP$, it suffices by polarization to verify the convergence when $P_1=\ldots=P_k=P$. In this case, let $Q_N=P - \E(P(B_t^N))=P -e^{t(D+\frac{1}{N^2}L)}P(\mx{1})$. We want to prove that
%$$N^k \left(e^{D+\frac{1}{N^2}L}(Q_N^k)\right)(\mx{1}) =\sum_{\pi\in\mathcal{P}_2(k)}\prod_{(i,j)\in\pi}  \sigma_{\mx{T}}(P,P)+O(1/N).$$

\begin{lemma} \label{l.Driver}
For all $t\geq 0$, we have\label{dev}
\begin{equation}e^{t(D+\frac{1}{N^2}L)}=e^{tD}+\frac{1}{N^2}\int_0^t e^{t_1(D+\frac{1}{N^2}L)}Le^{(t-t_1)D}\,dt_1.\end{equation}
In particular, $e^{D+\frac{1}{N^2}L}=e^D+\mx{O}(1/N^2).$ More generally, for all $k \in \N$, we have
 \begin{align*}e^{t(D+\frac{1}{N^2}L)}=&e^{tD}+\sum_{n=1}^{k}\frac{1}{N^{2n}} \int_{0\leq t_n\leq \cdots \leq t_1\leq t} e^{t_nD}Le^{(t_{n+1}-t_n)D}L\cdots Le^{(t-t_1)D}\,dt_1\cdots dt_{n}\\
 &+\frac{1}{N^{2(k+1)}}\int_{0\leq t_{k+1}\leq \cdots \leq t_1\leq t}e^{t_{k+1}(D+\frac{1}{N^2}L)}Le^{(t_{k+1}-t_k)D}L\cdots Le^{(t-t_1)D}\,dt_1\cdots dt_{k+1}.
  \end{align*}
\end{lemma}
\begin{proof}Let us define $S(t)=e^{t(D+\frac{1}{N^2}L)}e^{-tD}$; then $S$ is differentiable, and
\[S'(t)=e^{t(D+\frac{1}{N^2}L)}(D+\frac{1}{N^2}L-D)e^{-tD}=\frac{1}{N^2}S(t)e^{tD}Le^{-tD}. \] 
Since $S(0)=I_N$, it follows that $S(t)=1+\frac{1}{N^2} \int_0^t S(\theta)e^{\theta D}Le^{-\theta D}\,d\theta$. Multiplying by $e^{tD}$ on the right gives us the first formula. We can then compute
$$\vertiii{e^{D+\frac{1}{N^2}L}-e^D}\leq \frac{1}{N^2}\int_0^1 \vertiii{e^{t_1(D+\frac{1}{N^2}L)}Le^{(1-t_1)D}}\,dt\leq \frac{1}{N^2}e^{2\vertiii{D}+\vertiii{L}}\vertiii{L}.$$
The last formula is obtained by induction over $k$, using at each step the first formula. \end{proof}

\begin{remark} The first formula is often called {\em Duhamel's formula}: for any operators $A,B$ on some finite dimensional vector space $V$,
\[ e^A - e^B = \int_0^1 e^{sA}(A-B)e^{(1-s)B}\,ds. \]
The proof is the same as given above in the case $t=1$ (with $A=D+\frac{1}{N^2}L$ and $B=D$); conversely, the general case follows from Duhamel's formula by a simple change of variables.
\end{remark}

For $n\in\N$, let us denote by $\Delta_n\subset\R^n$ the simplex
\[ \Delta_n = \{(t_1,\ldots,t_n)\in\R^n\colon 0\le t_{n}\le t_{n-1}\le \cdots \le t_{1}\le 1\}. \]
Using Lemma \ref{l.Driver} at step $[k/2]$, the study of the limit of
$N^k \left[e^{D+\frac{1}{N^2}L}(Q_N^k)\right](\mx{1})$ is decomposed into the study of the limits of:
\begin{enumerate}
\item $N^k\left[ e^D(Q_N^k)\right](\mx{1})$,
\item $N^{k-2n} \left[\int_{\Delta_n} e^{t_nD}Le^{(t_{n+1}-t_n)D}L\cdots Le^{(1-t_1)D}\,dt_1\cdots dt_{n}(Q_N^k)\right](\mx{1})$ for $1\leq n \leq  [k/2]$, and
%\item $N^{k-2n}\int_{0\leq t_1\leq \cdots \leq t_{n}\leq 1} e^{t_1D}Le^{(t_2-t_1)D}L\cdots Le^{(1-t_{n})D}\,dt_1\cdots dt_{n}(Q_N^k)(\mx{1})$ if $ n =  [k/2]$
\item $N^{k-2-[k/2]}\left[\int_{\Delta_{k+1}}e^{t_{k+1}(D+\frac{1}{N^2}L)}Le^{(t_{k+1}-t_k)D}L\cdots Le^{(1-t_1)D}\,dt_1\cdots dt_{k+1}(Q_N^k)\right](\mx{1})$,
\end{enumerate}
which we address separately in the following three steps. In the fourth step, we sum up the three convergences to  conclude the proof.  We will see that the only term which does not vanish is the second term considered when $n=[k/2]$.

\paragraph*{Step 1}Since the map $A\mapsto [A(P)](\mx{1})$ is linear on $\End(\PP_{kd})$, it is therefore bounded and we deduce$$\left[e^{D+\frac{1}{N^2}L}(P)\right](\mx{1})=[e^D(P)](\mx{1})+O(1/N^2)$$from $e^{D+\frac{1}{N^2}L}=e^D+\mx{O}(1/N^2).$ But we have $Q_N=P-\E[P(B^N(\mx{T}))]=P-\left[e^{D+\frac{1}{N^2}L}(P)\right](\mx{1})$ thanks to Corollary~\ref{exbt}. Consequently, \begin{equation}
Q_N=P-[e^D(P)](\mx{1})+\textsc O(1/N^2)\label{qn}
\end{equation} and therefore $Q_N^k=(P-[e^D(P)](\mx{1}))^k+\textsc O(1/N^{2k})$. Since $D$ satisfies the product rule, we deduce from a standard formal power series argument that $e^{D}$ is an algebra homomorphism. Thus \begin{align*}
[e^D(Q_N^k)](\mx{1})&=\left[e^D(P-[e^D(P)](\mx{1}))^k\right](\mx{1})+O(1/N^{2k})\\&=(\left[e^D(P)(\mx{1})-[e^D(P)](\mx{1})\right]^k)(\mx{1})+O(1/N^{2k})\\&=([e^D(P)](\mx{1})-[e^D(P)](\mx{1}))^k+O(1/N^{2k})\\&=O(1/N^{2k}).
\end{align*}
Finally, $N^k\left[ e^D(Q_N^k)\right](\mx{1})=O(1/N^k)$.
\paragraph*{Step 2}We are assuming at this step that $2\leq k$.
%Let us compute for fixed $0\leq t_1\leq \cdots \leq t_{n}\leq 1$ the term $e^{t_1D}Le^{(t_2-t_1)D}L\cdots Le^{(1-t_{n})D}(Q_N^k)$.
For all $R\in \PP(J)$, $t\geq 0$ and $n\geq 2$, we have by Lemma~\ref{landg}\begin{align*} L((e^{tD}(Q_N))^n\cdot R)=&(e^{tD}Q_N)^nL(R)+2n(e^{tD}Q_N)^{n-1}\Gamma(e^{tD}Q_N,R)\\&
+n(e^{tD}Q_N)^{n-1}L(e^{tD}Q_N)R+n(n-1)(e^{tD}Q_N)^{n-2}\Gamma(e^{tD}Q_N,e^{tD}Q_N)R.\end{align*} In others words, for all $d'\leq (k-1)d$, if we define the bilinear map $B_n:(S,R)\mapsto S\cdot L(R)+2n\Gamma(S,R)+nL(S)\cdot R$ from $\PP_d\times \PP_{d'}$ to $\PP_{d+d'}$, we have, for all $R\in \PP_{d'}$, \begin{equation}L((e^{tD}(Q_N))^n\cdot R)=(e^{tD}Q_N)^{n-1}B_n(e^{tD}Q,R)+n(n-1)(e^{tD}Q_N)^{n-2}\Gamma(e^{tD}Q_N,e^{tD}Q_N)R.\label{LR}\end{equation}
Let us denote by $\Gamma(t)$ the element $e^{tD}\Gamma( e^{(1-t)D}Q_N, e^{(1-t)D}Q_N)\in \PP_{2d}$. Using~\eqref{LR}, we prove by induction on $n$ the following lemma.
\begin{lemma} For all $n$ such that $1\leq n \leq [k/2]$ and $0\leq t_n\leq \cdots \leq t_0= 1$, there exists $R_n\in \PP_{(2n-1)d}$ bounded independently of $N,t_0,\ldots t_n$ such that
\begin{multline}Le^{(t_{n-1}-t_n)D}L\cdots L e^{(1-t_1)D}(Q^k_N)\\
=\frac{k!}{(k-2n)!}(e^{(1-t_n)D}Q_N)^{k-2} e^{-t_nD}(\Gamma(t_1)\cdots \Gamma(t_n))+(e^{(1-t_n)D}Q_N)^{k-2n+1} R_n.\label{ind}
\end{multline}
\end{lemma}
\begin{proof}
Indeed, when $n=1$, setting $R_1=k L (e^{(1-t_1)D}Q_N)\in \PP_d$, we have$$ L e^{(1-t_1)D}(Q^k_N)
=k(k-1)(e^{(1-t_1)D}Q_N)^{k-2n} \Gamma( e^{(1-t_1)D}Q_N, e^{(1-t_1)D}Q_N)+(e^{(1-t_1)D}Q_N)^{k-1} R_1.
$$
Remark that $\|R_1\|\leq k\vertiii{L}e^{\vertiii{D}}\|Q_N\|$. Because of \eqref{qn},
\begin{equation}\|Q_N\|\leq \|P-[e^D(P)](\mx{1})\|+\|Q_N-P+[e^D(P)](\mx{1})\|=\|P-[e^D(P)](\mx{1})\|+O(1/N^2).\label{boundQN}\end{equation} Therefore, $Q_N$ is bounded independently of $N$, and so too is $R_1$.  Assume now that $2\leq n \leq [k/2]$ and that \eqref{ind} has been verified up to level $n-1$. We compute
\begin{align*}&\hspace{-0.5cm}L e^{(t_{n-1}-t_n)D}L \cdots L  e^{(1-t_1)D}(Q^k_N)\\
=&L e^{(t_{n-1}-t_n)D}\left(\frac{k!}{(k-2n+2)!}(e^{(1-t_{n-1})D}Q_N)^{k-2n+2} e^{-t_{n-1}D}(\Gamma(t_1)\cdots \Gamma(t_{n-1}))\right.\\
&\hspace{9cm}\left.\vphantom{\frac{k!}{(k-2n+2)!}}+(e^{(1-t_{n-1})D}Q_N)^{k-2n+3} R_{n-1}\right)\\
=&\frac{k!}{(k-2n+2)!}L \left((e^{(1-t_{n})D}Q_N)^{k-2n+2} e^{-t_{n}D}(\Gamma(t_1)\cdots \Gamma(t_{n-1}))\right)\\
&\hspace{9cm}+L \left((e^{(1-t_{n})D}Q_N)^{k-2n+3} R_{n-1}\right).
\end{align*}
We use now \eqref{LR} on each term. The first term leads to
\begin{align*}&\frac{k!}{(k-2n)!}(e^{(1-t_n)D}Q_N)^{k-2n} e^{-t_nD}(\Gamma(t_1)\cdots \Gamma(t_n))\\
+&\frac{k!}{(k-2n+2)!}(e^{(1-t_{n})D}Q_N)^{k-2n+1} B_{k-2n+2}\left(e^{(1-t_{n})D}Q_N,e^{-t_{n}D}(\Gamma(t_1)\cdots \Gamma(t_{n-1}))\right) ,
\end{align*}
and the second term to
\begin{align*}&(e^{(1-t_{n})D}Q_N)^{k-2n+2}B_{k-2n+3}(e^{(1-t_{n})D}Q_N,R_{n-1})\\
+&(k-2n+3)(k-2n+2)(e^{(1-t_{n})D}Q_N)^{k-2n+1}\Gamma(e^{(1-t_{n})D}Q_N,e^{(1-t_{n})D}Q_N)R_{n-1}.
\end{align*}
Thus, $R_n\in \PP_{(2n-1)d}$ can be defined by
\begin{align*}R_n &\equiv \frac{k!}{(k-2n+2)!} B_{k-2n+2}\left(e^{(1-t_{n})D}Q_N,e^{-t_{n}D}(\Gamma(t_1)\cdots \Gamma(t_{n-1})\right)\\
&\qquad +(e^{(1-t_{n})D}Q_N)B_{k-2n+3}(e^{(1-t_{n})D}Q_N,R_{n-1})\\
&\qquad+(k-2n+3)(k-2n+2)\Gamma(e^{(1-t_{n})D}Q_N,e^{(1-t_{n})D}Q_N)R_{n-1}
\end{align*}
which verifies~\eqref{ind} and which is bounded by
\begin{align*}&\frac{k!}{(k-2n+2)!} \vertiii{B_{k-2n+2}}_{(d,2(n-1)d)}e^{2 \vertiii{D}}\|Q_N\|\|\Gamma(t_1)\|\cdots \|\Gamma(t_{n-1})\|\\
+&e^{2 \vertiii{D}}\|Q_N\|^2 \vertiii{B_{k-2n+3}}_{(d,(2n-1)d}\|R_{n-1}\|\\
+&(k-2n+3)(k-2n+2)\vertiii{\Gamma}_{(d,d)}e^{2 \vertiii{D}}\|Q_N\|^2\|R_{n-1}\|.
\end{align*}
Because of \eqref{boundQN}, it is bounded independently of $N$. We deduce also that $$\Gamma(t_i)=e^{t_iD}\Gamma( e^{(1-t_i)D}Q_N, e^{(1-t_i)D}Q_N)$$ is bounded by $\vertiii{\Gamma}_{(d,d)}e^{2\vertiii{D}}\|Q_N\|^2$ and consequently is bounded independently of $N,t_1,\ldots, t_n$.  Thus, $R_n$ is bounded independently of $N,t_1,\ldots, t_n$, as required.
\end{proof}

We recall that, because $D$ is a first order operator, $e^{t_nD}$ is  an algebra homomorphism.  Applying $e^{t_nD}$ to \eqref{ind} on the left, we obtain that, for all $1\leq n \leq [k/2]$,  $N\in \mathbb{N}$, and $(t_1,\ldots,t_n)\in \Delta_n$, there exists $R_n\in \PP_{(2n-1)d}$ bounded uniformly in $N,t_0,\ldots t_n$ such that $$e^{t_nD}Le^{(t_{n-1}-t_n)D}L\cdots L e^{(1-t_1)D}(Q^k_N)\\
=\frac{k!}{(k-2n)!}(e^{D}Q_N)^{k-2n} \Gamma(t_1)\cdots \Gamma(t_n)+(e^{D}Q_N)^{k-2n+1} R_n,$$
where $\Gamma(t)$ denotes the element $e^{tD}\Gamma( e^{(1-t)D}Q_N, e^{(1-t)D}Q_N)\in \PP_{2d}$.

From \eqref{qn}, we deduce that we have $\left[(e^DQ_N)^{k-2n}\right](\mx{1})=O(1/N^{2k-4n})$ and $\left[(e^DQ_N)^{k+1-2n}\right](\mx{1})=O(1/N^{2k+1-4n})$. We have already remarked in the proof of~\eqref{ind} that $\Gamma(t_i)=e^{t_iD}\Gamma( e^{(1-t_i)D}Q_N, e^{(1-t_i)D}Q_N)$ and $Q_N$ are bounded independently of $N,t_1,\ldots, t_n$; consequently, $$N^{2k-4n}\frac{k!}{(k-2n)!}\left[(e^DQ_N)^{k-2n} \Gamma(t_1)\cdots \Gamma(t_n)\right](\mx{1})\text{ and }N^{k+1-2n}\left[(e^DQ_N)^{k+1-2n} R_n\right](\mx{1})$$ are bounded independently of $N,t_1,\ldots, t_n$, and we deduce that
$$N^{k-2n}\left[\int_{\Delta_n} e^{t_nD}Le^{(t_{n+1}-t_n)D}L\cdots Le^{(t-t_1)D}\,dt_1\cdots dt_{n}(Q_N^k)\right](\mx{1})$$
is $O(1/N)$ if $k> 2n$ and is equal to $k!\int_{\Delta_n} \left(\Gamma(t_1)\cdots \Gamma(t_n)\right)(\mx{1})\,dt_1\cdots dt_{n}+O(1/N)$
if $k=2n$.

In the case where $k=2n$, because the integrand is symmetric in $t_1,\ldots, t_n$, the remaining term is equal to 
\[ \frac{k!}{n!}\int_{0\leq t_1, \ldots , t_{n}\leq 1}\left[\Gamma(t_1)\cdots \Gamma(t_n)\right](\mx{1})\,dt_1\cdots dt_n
= \frac{k!}{n!}\left(\int_0^1 \left[\Gamma(t)\right](\mx{1})\,dt\right)^n=\frac{(2n)!}{2^nn!} \sigma_{\mx{T}}(Q_N,Q_N)^n.\]
Note that $L$ kills constants, and similarly $\Gamma(P+c,Q+d) = \Gamma(P,Q)$ for any $c,d\in\C$. As a consequence, $ \sigma_{\mx{T}}(Q_N,Q_N)= \sigma_{\mx{T}}(P,P)$.

To sum up, $$N^{k-2n}\left[\int_{\Delta_n} e^{t_nD}Le^{(t_{n+1}-t_n)D}L\cdots Le^{(t-t_1)D}\,dt_1\cdots dt_{n}(Q_N^k)\right](\mx{1})$$
is equal to $\frac{(2n)!}{2^nn!} \sigma_{\mx{T}}(P,P)^n+O(1/N)$ if $k=2n$ and $O(1/N)$ if not.

\paragraph*{Step 3}We have $Q_N^k=(P-[e^D(P)](\mx{1}))^k+O(1/N^{2k})$ and
$$\vertiii{\int_{\Delta_{k+1}}e^{t_{k+1}(D+\frac{1}{N^2}L)}Le^{(t_{k+1}-t_k)D}L\cdots Le^{(t-t_1)D}\,dt_1\cdots dt_{k+1}} \leq \vertiii{L}^n e^{\vertiii{L}+\vertiii{D}}.$$
Consequently $$\left[\int_{\Delta_{k+1}}e^{t_{k+1}(D+\frac{1}{N^2}L)}Le^{(t_{k+1}-t_k)D}L\cdots Le^{(1-t_1)D}\,dt_1\cdots dt_{k+1}(Q_N^k)\right](\mx{1})$$ is bounded independently of $N$. On the other hand, $k-2([k/2]+1)\leq -1$ and $N^{k-2([k/2]+1)}$ is therefore $O(1/N)$. Thus, the term studied is $O(1/N)$.

\paragraph*{Step 4}Finally, applying Lemma~\ref{l.Driver} with $n=[k/2]$, and using the limits computed in the three previous steps, we have $N^k \E(Q_N^k(B^N(\mx{T})))=\frac{k!}{2^{k/2}(k/2)!} \sigma_{\mx{T}}(P,P)^{k/2}+O(1/N)$ if $k$ is even and $O(1/N)$ if not. Because the cardinality of $\mathcal{P}_2(k)$ is $\frac{k!}{2^{k/2}(k/2)!}$ if $k$ is even and $0$ if not, we have demonstrated the desired bound,
\[ N^k \E(Q_N^k(B^N(\mx{T}))) = \sum_{\pi\in\mathcal{P}_2(k)}\prod_{(i,j)\in\pi}  \sigma_{\mx{T}}(Q,Q) + O\left(\frac{1}{N}\right). \]
This concludes the proof of Theorem \ref{main}. \hfill $\square$

\section{Study of the covariance}\label{studycov}

In~\cite{LevyMaida2010}, L\'evy and Ma\"ida established a central limit theorem for random matrices arising from a unitary Brownian motion, which corresponds to the $(r,s)=(1,0)$ case.
\begin{theorem}{\cite[Theorem 2.6]{LevyMaida2010}} Let $(U_t^{N})_{t\geq 0}$ be a unitary Brownian motion on $\U_N$ ($U^N(t) = B^N_{1,0}(t)$ in our language). Let $P_1,\cdots,P_n\in \C[X,X^{-1}]$, and $T\geq 0$. When $N\to\infty$, the random vector
$$N\left(\tr \left(P_i(U^{N}(T))\right)-\E\left[\tr \left(P_i(U^{N}(T))\right)\right]\right)_{1\leq i \leq n} $$
converges in distribution to a Gaussian vector.
\end{theorem}
\noindent (In fact, the test functions allowed in their approach were not only polynomials but $C^1$ real-valued functions with Lipschitz derivative on the unit circle.  Generalizing to $\GL_N$ does not allow for such functional calculus. The statement above for Laurent polynomials is obtained easily from the real-valued case by linearity.)

The limit covariance involves three free unitary Brownian motion $(u_t)_{t\geq 0}$, $(v_t)_{t\geq 0}$ and $(w_t)_{t\geq 0}$ which are freely independent (cf.\ Section \ref{s.fmbm}, in the special case $(r,s)=(1,0)$).
Following~\cite{LevyMaida2010}, for all $P\in \C[X,X^{-1}]$, we denote by $P'\in \C[X,X^{-1}]$ the derivative of $P$ on the unit circle:
\[P'(z)=\lim_{h\to 0}\frac{f(ze^{ih})-f(z)}{h}, \quad \text{ for }\; z\in \U. \]
(Concretely, for all $n\in \mathbb{Z}$, if $P=X^n$ then $P'=inX^n$.) L\'evy and Ma\"ida proved that, for all $P,Q\in \C[X,X^{-1}]$, the covariance of the random variables $N(\tr P(U_{T}^{N})-\E[\tr P(U_{T}^{N})])$ and $N(\tr Q(U_{T}^{N})-\E[\tr Q(U_{T}^{N})])$ is asymptotically equal to
\begin{equation}
%\sigma_T(P,Q^*) =
\int_0^T \tau\left(P'(u_tv_{T-t})(Q'(u_tw_{T-t}))^*\right)dt, \label{cococovariance}\end{equation}
and moreover that, as $T\to\infty$, this approaches the Sobolev $H_{1/2}$ inner product of $P,Q$ (cf.\ \cite[Theorem 9.3]{LevyMaida2010}). Note that the expression \eqref{cococovariance} is obtained from the expression of the covariance in \cite[Definition 2.4]{LevyMaida2010} by linearity (since the expression of the covariance in \cite[Definition 2.4]{LevyMaida2010} is only valid for real-valued functions).

In this section, we relate our result to theirs by giving another expression of the covariance of the fluctuations of the more general processes $B^N_{r,s}$, which naturally generalizes \eqref{cococovariance}.

\subsection{New characterization of the covariance}
Denote by $I_N^J$ the identity element $(I_N,\ldots,I_N)\in \GL_N^J$. In the following proposition, we express the covariance with the help of three independent $(r,s)$-Brownian motions.
\begin{proposition}Let $B^N,C^N,D^N$ be three families of independent $(r,s)$-Brownian motions on $\GL_N$ indexed by $J$ which are independent. For all $P,Q\in \PP(J)$, we have\label{covn}
\[ \sigma_{\mx{T}} (P,Q) =2 N^2\int_0^1\E\left[\Gamma^{\mx{T}}_N\left([P]_N(B^N_{t\mx{T}}(\cdot)C^N_{(1-t)\mx{T}}),[Q]_N(B^N_{t\mx{T}}(\cdot)D^N_{(1-t)\mx{T}})\right)(I_N^J)\right]dt+O\left(\frac{1}{N^2}\right). \]
\end{proposition}
\noindent (To be clear on notation: the functions in the arguments of $\Gamma_N^\mx{T}$ above are
\[ G\mapsto [P]_N(B^N_{t\mx{T}}GC^N_{(1-t)\mx{T}}) \quad\text{and}\quad G\mapsto [Q]_N(B^N_{t\mx{T}}GD^N_{(1-t)\mx{T}}); \]
the resultant function after applying $\Gamma_N^\mx{T}$ is then evaluated at $I_N^J$ before integrating.  This $(\cdot)$ notation is used throughout this section.)

\begin{proof}For all $P,Q\in\PP(J)$, we have
\[ \sigma_{\mx{T}}(P,Q) = 2\int_0^1 \left[e^{t\D^{\mx{T}}}\left(\Gamma^{\mx{T}}(e^{(1-t)\D^{\mx{T}}}P,e^{(1-t)\D^{\mx{T}}}Q)\right)\right](\mx{1})\,dt. \]
As in the proof of Theorem~\ref{main}, we restrict our computations on a finite-dimensional space $\PP_d$ (take $d$ to be the sum of the degrees of $P$ and $Q$). Because of Lemma~\ref{dev}, we have $N^2\left(e^{t\D^{\mx{T}}}-e^{t(\D^{\mx{T}}+\frac{1}{N^2}\L^{\mx{T}})}\right)$ bounded independently of $N$ and $t$; consequently, it is straightforward to verify that
\[ \sigma_{\mx{T}}(P,Q) = 2\int_0^1 \left[e^{t(\D^{\mx{T}}+\frac{1}{N^2}\L^{\mx{T}})}\left(\Gamma^{\mx{T}}(e^{(1-t)(\D^{\mx{T}}+\frac{1}{N^2}\L^{\mx{T}})}P,e^{(1-t)(\D^{\mx{T}}+\frac{1}{N^2}\L^{\mx{T}})}Q)\right)\right](\mx{1})\,dt+O\left(\frac{1}{N^2}\right). \]
Hence, the proof will be complete once we show that, for $0\leq t\leq 1$,
\begin{multline}  \left[e^{t(\D^{\mx{T}}+\frac{1}{N^2}\L^{\mx{T}})}\left(\Gamma^{\mx{T}}(e^{(1-t)(\D^{\mx{T}}+\frac{1}{N^2}\L^{\mx{T}})}P,e^{(1-t)(\D^{\mx{T}}+\frac{1}{N^2}\L^{\mx{T}})}Q)\right)\right](\mx{1})\\=N^2\E\left[\Gamma^{\mx{T}}_N\left([P]_N(B^N_{t\mx{T}}(\cdot)C^N_{(1-t)\mx{T}}),[Q]_N(B^N_{t\mx{T}}(\cdot)D^N_{(1-t)\mx{T}})\right)(I_N^J)\right].\label{aime}
\end{multline}
Fix $t\in[0,1]$. We start from the left side to recover the right side. First of all, using Theorem~\ref{deltaonp}, Proposition~\ref{p.carre} and Definition~\ref{genbnt}, we have
\begin{align*}
&\left[e^{t(\D^{\mx{T}}+\frac{1}{N^2}\L^{\mx{T}})}\left(\Gamma^{\mx{T}}(e^{(1-t)(\D^{\mx{T}}+\frac{1}{N^2}\L^{\mx{T}})}P,e^{(1-t)(\D^{\mx{T}}+\frac{1}{N^2}\L^{\mx{T}})}Q)\right)\right](\mx{1})\\
=&e^{t(\mx{T}\cdot\Delta^N)}\left[\Gamma^{\mx{T}}(e^{(1-t)(\D^{\mx{T}}+\frac{1}{N^2}\L^{\mx{T}})}P,e^{(1-t)(\D^{\mx{T}}+\frac{1}{N^2}\L^{\mx{T}})}Q)\right]_N(I_N^J)\\
=&N^2 \left[e^{t(\mx{T}\cdot\Delta^N)}\left(\Gamma^{\mx{T}}_N([e^{(1-t)(\D^{\mx{T}}+\frac{1}{N^2}\L^{\mx{T}})}P]_N,[e^{(1-t)(\D^{\mx{T}}+\frac{1}{N^2}\L^{\mx{T}})}Q]_N)\right)\right](I_N^J)\\
=&N^2 \left[e^{t(\mx{T}\cdot\Delta^N)}\left(\Gamma^{\mx{T}}_N(e^{(1-t)(\mx{T}\cdot\Delta^N)}[P]_N,e^{(1-t)(\mx{T}\cdot\Delta^N)}[Q]_N)\right)\right](I_N^J)\\
=&N^2\E\left[\left(\Gamma^{\mx{T}}_N(e^{(1-t)(\mx{T}\cdot\Delta^N)}[P]_N,e^{(1-t)(\mx{T}\cdot\Delta^N)}[Q]_N)\right)(B^N_{t\mx{T}})\right].
 \end{align*}
Recall that $\beta_{r,s}^N$ is an orthonormal basis of $\M_N$ for the metric $\langle\cdot,\cdot\rangle_{r,s}^N$, and for all $\xi\in \beta_{r,s}^N$, $\xi_j$ is the left-invariant vector field which acts only on the $j$th component of $\GL_N^J$. Let us denote respectively by $L_\xi$ and $R_\xi$ the left and the right translation by $\xi$. We follow the notation of Section~\ref{not}. We compute
 \begin{align*}
&\left(\Gamma^{\mx{T}}_N(e^{(1-t)(\mx{T}\cdot\Delta^N)}[P]_N,e^{(1-t)(\mx{T}\cdot\Delta^N)}[Q]_N\right)(B^N_{t\mx{T}})\\
=&\frac{1}{2}\left[\sum_{\xi\in \beta_{r,s}^N, j\in J} t_j\left(\partial_{\xi_j} (e^{(1-t)(\mx{T}\cdot\Delta^N)}[P]_N)\right)\left(\partial_{\xi_j}(e^{(1-t)(\mx{T}\cdot\Delta^N)}[Q]_N)\right)\right](B^N_{t\mx{T}})\\
=&\frac{1}{2}\sum_{\xi\in \beta_{r,s}^N, j\in J}t_j \left[\left(L_{B^N_{t\mx{T}}}\circ (\partial_{\xi_j} e^{(1-t)(\mx{T}\cdot\Delta^N)}[P]_N)\right)(I_N^J)\right]\cdot\left[\left(L_{B^N_{t\mx{T}}}\circ (\partial_{\xi_j} e^{(1-t)(\mx{T}\cdot\Delta^N)}[Q]_N)\right)(I_N^J)\right].
 \end{align*}
Here, in order to reverse the different operators, we introduce the right-invariant vector fields. For all $\xi\in \M_N$, let us denote by $\partial_\xi'$ the associated right-invariant vector field on $\GL_N$:
$$(\partial_\xi' f)(g)=\left.\frac{d}{dt}\right|_{t=0}f(e^{t\xi}g),\hspace{1cm} f\in C^{\infty}(\GL_N).$$
For all $\xi\in \beta_{r,s}^N$, $\del_{\xi_j}'$ is the corresponding right-invariant vector field which acts only on the $j$th component of $\GL_N^J$. Note that, for all  $\xi,\zeta\in \M_N$ and all $j,k\in J$, the vector fields $\partial_{\xi_i}$ and $\partial_{\zeta_k}'$ commute. Moreover, for all $f\in C^{\infty}(\GL_N^J)$, we have $(\partial_{\xi_j} f)(I_N^J)=(\partial_{\xi_j}' f)(I_N^J)$. Using those two facts, we compute
 \begin{align*}
\left(L_{B^N_{t\mx{T}}}\circ (\partial_{\xi_j} e^{(1-t)(\mx{T}\cdot\Delta^N)}[P]_N)\right)(I_N^J)
=&\left( \partial_{\xi_j} e^{(1-t)(\mx{T}\cdot\Delta^N)}(L_{B^N_{t\mx{T}}}\circ[P]_N)\right)(I_N^J)\\
=&\left( \partial_{\xi_j}' e^{(1-t)(\mx{T}\cdot\Delta^N)}(L_{B^N_{t\mx{T}}}\circ[P]_N)\right)(I_N^J)\\
=&\left( e^{(1-t)(\mx{T}\cdot\Delta^N)}\partial_{\xi_j}' (L_{B^N_{t\mx{T}}}\circ[P]_N)\right)(I_N^J)\\
=&\E\left[\left.\left( \partial_{\xi_j}' (L_{B^N_{t\mx{T}}}\circ[P]_N)\right)(C^N_{(1-t)\mx{T}})\right|B^N_{t\mx{T}}\right]\\
=&\E\left[\left.\left(R_{C^N_{(1-t)\mx{T}}} \circ\partial_{\xi_j}' (L_{B^N_{t\mx{T}}}\circ[P]_N)\right)(I_N^J)\right|B^N_{t\mx{T}}\right]\\
=&\E\left[\left.\left(\partial_{\xi_j}' (R_{C^N_{(1-t)\mx{T}}} \circ L_{B^N_{t\mx{T}}}\circ[P]_N)\right)(I_N^J)\right|B^N_{t\mx{T}}\right]\\
=&\E\left[\left.\partial_{\xi_j} \left([P]_N(B^N_{t\mx{T}}(\cdot)C^N_{(1-t)\mx{T}})\right)(I_N^J)\right|B^N_{t\mx{T}}\right]\\
 \end{align*}
 and similarly $\left(L_{B^N_{t\mx{T}}}\circ (\partial_{\xi_j} e^{(1-t)(\mx{T}\cdot\Delta^N)}[Q]_N)\right)(I_N^J)=\E\left[\left.\partial_{\xi_j} \left([Q]_N(B^N_{t\mx{T}}(\cdot)D^N_{(1-t)\mx{T}})\right)(I_N^J)\right|B^N_{t\mx{T}}\right]$. It follows that
  \begin{align*}
&\frac{1}{2}\sum_{\xi\in \beta_{r,s}^N, j\in J} t_j\left[\left(L_{B^N_{t\mx{T}}}\circ (\partial_{\xi_j} e^{(1-t)(\mx{T}\cdot\Delta^N)}[P]_N)\right)(I_N^J)\right]\cdot\left[\left(L_{B^N_{t\mx{T}}}\circ (\partial_{\xi_j} e^{(1-t)(\mx{T}\cdot\Delta^N)}[Q]_N)\right)(I_N^J)\right]\\
=&\frac{1}{2}\sum_{\xi\in \beta_{r,s}^N, j\in J}t_j\E\left[\left.\partial_{\xi_j} \left([P]_N(B^N_{t\mx{T}}(\cdot)C^N_{(1-t)\mx{T}})\right)(I_N^J)\cdot \partial_{\xi_j} \left([Q]_N(B^N_{t\mx{T}}(\cdot)D^N_{(1-t)\mx{T}})\right)(I_N^J)\right|B^N_{t\mx{T}}\right]\\
=&\E\left[\left.\Gamma^{\mx{T}}_N\left([P]_N(B^N_{t\mx{T}}(\cdot)C^N_{(1-t)\mx{T}}),[Q]_N(B^N_{t\mx{T}}(\cdot)D^N_{(1-t)\mx{T}})\right)(I_N^J)\right|B^N_{t\mx{T}}\right].
 \end{align*}
 Taking the expectation leads to the right side of \eqref{aime}.
\end{proof}
We shall now let the dimension tend to infinity in the previous proposition in order to have a new expression of the covariance involving three freely independent free multiplicative $(r,s)$-Brownian motions.
\begin{theorem} \label{t.main.cov} For all $P,Q\in \PP(J)$, there exists $\tilde{\Gamma}^{\mx{T}}(P,Q)\in \PP(J^3)$ such that for all $N\in \mathbb{N}$, and all $B,C,D\in \GL_N^J$,\label{cov}
\begin{equation}
N^2\Gamma^{\mx{T}}_N\left([P]_N(B(\cdot)C),[Q]_N(B(\cdot)D)\right)(I_N^J)=\left[\tilde{\Gamma}^{\mx{T}}(P,Q)\right]_N(B,C,D)\label{gamma}
\end{equation}
and in this case, taking three families $b,c,d$ of free multiplicative $(r,s)$-Brownian motions indexed by $J$ which are freely independent in a noncommutative probability space $(\mathscr{A},\tau)$, we have
$$ \sigma_{\mx{T}} (P,Q) =2\int_0^1\left[\tilde{\Gamma}^{\mx{T}}(P,Q)\right]_{(\mathscr{A},\tau)}(b_{t\mx{T}},c_{(1-t)\mx{T}},d_{(1-t)\mx{T}})dt.$$
\end{theorem}
\noindent This expression for the covariance, albeit instructive, is not explicit, but in the next section, we will compute the function $\left[\tilde{\Gamma}^{\mx{T}}(P,Q)\right]_N$ explicitly in the simple case $J=\{1\}$ and ${\mx{T}}=(T)$.

\begin{proof}Let us suppose first that the polynomials $P$ and $Q$ are given by $P=v_\ex$ and $Q=v_\delta$, with $\ex=((j_1,\ex_1),\ldots,(j_n,\ex_n))\in \EX$ and $\delta=((k_1,\delta_1),\ldots,(k_m,\delta_m))\in \EX$. Hence, for any input $G\in\M_N$,
$$[P]_N(BGC)=\tr((BGC)^{\ex_1}_{j_1}\cdots (BGC)^{\ex_n}_{j_n}),\quad [Q]_N(BGD)=\tr((BGD)^{\delta_1}_{h_1}\cdots (BGD)^{\delta_n}_{h_n}).$$
We can then compute
\begin{align*} \partial_{\xi_j} \left([P]_N(B(\cdot)C)\right)(I_N^J)
=&\sum_{l=1}^n\delta_{j,j_l} \tr((BC)^{\ex_1}_{j_1}\cdots (B\xi C)^{\ex_k}_{j_k}\cdots(BC)^{\ex_n}_{j_n})\\
=&\sum_{l=1}^n\delta_{j,j_l} \tr(\xi^{\ex_k} \cdot [v_{\ex^{(l)}}]_N(B,C,D)),
\end{align*}
where $\ex^{(l)}$ is a word in $\EX(J^3)$, which depends on $\ex$ and $l$. Similarly,
$$\partial_{\xi_j} \left([Q]_N(B(\cdot)D)\right)(I_N^J)=\sum_{h=1}^m\delta_{j,j_h} \tr(\xi^{\delta_h}\cdot [v_{\delta^{(h)}}]_N(B,C,D)),$$
where $\delta^{(h)}$ is a word in $\EX(J^3)$, which depends on $\delta$ and $h$. Finally, using the magic formula of Proposition~\ref{magicformula}, we have
\begin{multline*}\frac{N^2}{2}\sum_{\xi\in \beta_{r,s}^N, j\in J}t_j\partial_{\xi_j} \left([P]_N(B(\cdot)C)\right)\partial_{\xi_j} \left([Q]_N(B(\cdot)D)\right)(I_N^J)\\=\frac{1}{2}\sum_{ j\in J}t_j\sum_{l=1}^n\sum_{h=1}^m\delta_{j,j_l}\delta_{j,j_h}\left(s+\sigma_{l,h} r\right)\tr\left( [v_{\ex^{(l)}}](B,C,D)\cdot [v_{\delta^{(h)}}]_N(B,C,D)\right),
\end{multline*}
where $\sigma_{l,h}\in\{\pm1\}$ depends on $\epsilon$, $\delta$, $l$ and $h$.
Thus, the element $$\tilde{\Gamma}^{\mx{T}}(v_\ex,v_\delta)=\frac{1}{2}\sum_{ j\in J}t_j\sum_{l=1}^n\sum_{h=1}^m\delta_{j,j_l}\delta_{j,j_h}\left(r+\sigma_{l,h} s\right)v_{\ex^{(l)}\delta^{(h)}}\in \PP(J^3)$$ satisfies~\eqref{gamma}.

We extend the definition of $\tilde{\Gamma}^{\mx{T}}$ to all elements of $\PP(J)$ of the form $P_1\cdots P_k,Q_1\cdots Q_l \in \PP$ by the relation
$$\tilde{\Gamma}^{\mx{T}}(P_1\cdots P_k,Q_1\cdots Q_l)=\sum_{1\leq i\leq k}\sum_{1\leq j\leq l} P_1\cdots \widehat{P_i}\cdots P_kQ_1\cdots \widehat{Q_j}\cdots Q_l\tilde{\Gamma}^{\mx{T}}(P_i,Q_j),$$
and finally, we extend $\tilde{\Gamma}^{\mx{T}}$ to all elements of $\PP(J)$ by bilinearity. Because $\Gamma^{\mx{T}}_N$ fulfills the same relations, this demonstrates \eqref{gamma}.

Thanks to Proposition~\ref{covn}, we have
\begin{align*} \sigma_{\mx{T}} (P,Q)
&=2 N^2\int_0^1\E\left[\Gamma^{\mx{T}}_N\left([P]_N(B^N_{t\mx{T}}(\cdot)C^N_{(1-t)\mx{T}}),[Q]_N(B^N_{t\mx{T}}(\cdot)D^N_{(1-t)\mx{T}})\right)(I_N^J)\right]dt+O\left(\frac{1}{N^2}\right)\\
&=2\E\left[\int_0^1\left[\tilde{\Gamma}^{\mx{T}}(P,Q)\right]_N(B^N_{t\mx{T}},C^N_{(1-t)\mx{T}},D^N_{(1-t)\mx{T}})dt \right]+O\left(\frac{1}{N^2}\right)\\
&=2\E\left[\left[ \int_0^1\left(\tilde{\Gamma}^{\mx{T}}(P,Q)\right)dt\right]_N (B^N_{t\mx{T}},C^N_{(1-t)\mx{T}},D^N_{(1-t)\mx{T}})\right]+O\left(\frac{1}{N^2}\right).\\
\end{align*}
In~\cite{Kemp2013b}, it is proved that, for all $R\in \PP(J^3)$, we have
\[ \E\left[ [R]_N(B^N_{t\mx{T}},C^N_{(1-t)\mx{T}},D^N_{(1-t)\mx{T}})\right]=[R]_{(\mathscr{A},\tau)}(b_{t\mx{T}},c_{(1-t)\mx{T}},d_{(1-t)\mx{T}})+O\left(\frac{1}{N^2}\right), \] 
(cf.\ Remark \ref{r.copout}). Letting $N\to\infty$, it follows that
\begin{equation*}
\sigma_{\mx{T}} (P,Q) =2\int_0^1\left[\tilde{\Gamma}^{\mx{T}}(P,Q)\right]_{(\mathscr{A},\tau)}(b_{t\mx{T}},c_{(1-t)\mx{T}},d_{(1-t)\mx{T}})dt.\qedhere
\end{equation*}
\end{proof}

\begin{remark} \label{r.copout} The main theorem \cite[Theorem 1.1]{Kemp2013b} is stated in the special case that $R$ is the trace of a noncommutative polynomial, and moreover only for instances of a single Brownian motion.  However, \cite[Corollary 5.6]{Kemp2013b} shows how to quickly and easily extend this to the more general setting of convergence of any trace polynomial in instances of any finite family of independent Brownian motions, as we use presently. \end{remark}

% To be clear, the main theorem \cite[Theorem 1.6]{Kemp2013b} yields this result in the special case that $R$ is the trace of a noncommutative polynomial.  (It is stated for instances of a single Brownian motion at any collection of times; because the inverse of the Brownian motion is also a Brownian motion, and the multiplicative increments are independent, it follows immediately that this convergence holds for any collection of independent Brownian motions.)  The full statement involving any trace polynomial $R$ follows from a similar proof, with few modifications; the key covariance estimate \cite[Theorem 1.13]{Kemp2013b} is just a special case of \cite[Theorem 3.16]{Kemp2013b}, which indeed holds for all trace polynomials $R$ as needed.  \end{remark}

\subsection{The simple case of polynomials}

Throughout this section, we investigate the case where $J=\{1\}$ and ${\mx{T}}=(T)$. In this case, we have the injective map from $\C\langle X,X^*\rangle$ to $\tr(\C\{J\})\cong \PP(J)$ denoted by $\tr$, and similarly the injective map from $\C\langle X_1,X_1^*,X_2,X_2^*,X_3,X_3^*\rangle$ to $\tr(\C\{J^3\})\cong\PP(J^3)$ also denoted by $\tr$.

In the case of a polynomial, it is possible to compute explicitly the term $\left[\tilde{\Gamma}^\mx{T}(P,Q)\right]$ of Theorem~\ref{cov}, and thus recover the expression for the covariance given by~\eqref{cococovariance}.

\begin{corollary} \label{c.main.cov} Let us suppose $J=\{1\}$, ${\mx{T}}=(T)$, and $P, Q\in \C[X]$. Then, following Theorem~\ref{cov},
$$\left(\tilde{\Gamma}^{\mx{T}}(\tr P,\tr Q)\right)=\frac{T}{2}(s-r)\tr(P'(X_1X_2)Q'(X_1X_3)) , $$
$$\left(\tilde{\Gamma}^{\mx{T}}(\tr P,\tr Q^*)\right)=\frac{T}{2}(r+s)\tr(P'(X_1X_2)(Q'(X_1X_3))^*)  $$
and
$$\left(\tilde{\Gamma}^{\mx{T}}(\tr P^*,\tr Q^*)\right)=\frac{T}{2}(s-r)\tr((P'(X_1X_2))^*(Q'(X_1X_3))^*)  .$$
Consequently, taking three free multiplicative $(r,s)$-Brownian motions $b,c,d$ which are freely independent in a noncommutative probability space $(\mathscr{A},\tau)$, we have
\begin{align*} \sigma_{\mx{T}} (\tr P,\tr Q) &=(s-r)\int_0^T\tau\!\left[P'(b_t c_{T-t})Q'(b_t d_{T-t}))\right]dt, \\
\sigma_{\mx{T}} (\tr P,\tr Q^*) &=(r+s)\int_0^T\tau\!\left[P'(b_t c_{T-t})(Q'(b_t d_{T-t}))^*\right]dt, \quad \text{and} \\
\sigma_{\mx{T}} (\tr P^*,\tr Q^*) &=(s-r)\int_0^T\tau\!\left[(P'(b_t c_{T-t}))^*(Q'(b_t d_{T-t}))^*\right]dt.
\end{align*}
\end{corollary}

\begin{remark}
Let us make a few comments on this final corollary.
\begin{itemize}
\item[(1)] In the case $(r,s)=(1,0)$, this result shows that, for all $P,Q\in \C[X]$, the covariance of the random variables $N(\tr P(U_{T}^{N})-\E[\tr P(U_{T}^{N})])$ and $N(\tr Q(U_{T}^{N})-\E[\tr Q(U_{T}^{N})])$ is asymptotically equal to $\sigma_{\mx{T}} (\tr P,\tr Q^*)$, which reproduces exactly the expression of~\eqref{cococovariance} found by L\'evy and Ma\"ida in \cite{LevyMaida2010}.

\item[(2)] In the case $(r,s)=(\frac12,\frac12)$, this result shows that, for all $P\in \C[X]$, the fluctuation random variable $N\left(\tr(P(G^{N}_T))-\E\left[\tr(P(G^{N}_T))\right]\right)$ is asymptotically a circularly-symmetric complex normal distribution of variance $\int_0^T\tau(P'(b_t c_{T-t})(P'(b_td_{T-t}))^*)\,dt$, where $b,c,d$ are three freely independent standard free multiplicative Brownian motions.
\end{itemize}
\end{remark}

\begin{proof}
Let $P=X^n$ and $Q=X^m$. We have $\tr P=v_\ex$ and $\tr Q=v_\delta$ with $\ex=\overbrace{((1,1),\ldots,(1,1))}^{n\ \text{times}}\in \EX$ and $\delta=\overbrace{((1,1),\ldots,(1,1))}^{m\ \text{times}}\in \EX$. Let $N\in \mathbb{N}$, and $B,C,D\in \GL_N^J$. Then for all $G\in\M_N$,
\[ [\tr P]_N(CGB)=\tr((CGB)^n),\quad [\tr Q]_N(DGB)=\tr((DGB)^{m}). \]
We compute for all $\xi\in \beta_{r,s}^N$
$$\partial_{\xi} \left([\tr P]_N(B(\cdot)C)\right)(I_N^J)=n \tr(\xi (CB)^n)$$
and 
$$\partial_{\xi} \left([\tr Q]_N(B(\cdot)D)\right)(I_N^J)=m \tr(\xi (DB)^m).$$
Finally, using the magic formula of Proposition~\ref{magicformula}, we have
\begin{align*}\frac{N^2}{2}\sum_{\xi\in \beta_{r,s}^N}T\partial_{\xi} \left([\tr P]_N(B(\cdot)C)\right)\partial_{\xi} \left([\tr Q]_N(B(\cdot)D)\right)(I_N)&=\frac{T}{2}(s-r)mn \tr((CB)^n(DB)^m) \\
&= \left(\tilde{\Gamma}^{\mx{T}}(\tr P,\tr Q)\right)(B,C,D)\\
\end{align*}
with $\left(\tilde{\Gamma}^{\mx{T}}(\tr P,\tr Q)\right)=\frac{T}{2}(s-r)\tr(P'(X_1X_2)Q'(X_1X_3)) $. Similar computations lead to $\left(\tilde{\Gamma}^{\mx{T}}(\tr P,\tr Q^*)\right)=\frac{T}{2}(r+s)\tr(P'(X_1X_2)Q'(X_1X_3)^*)  $ and $\left(\tilde{\Gamma}^{\mx{T}}(\tr P^*,\tr Q^*)\right)=\frac{T}{2}(s-r)\tr(P'(X_1X_2)^*Q'(X_1X_3)^*)  ,$ and we extend the formulas to $P, Q\in \C[X]$ by bilinearity.

Thanks to Proposition~\ref{cov}, we know that
\begin{align*} \sigma_{\mx{T}} (\tr P,\tr Q) &=2\int_0^1\left(\tilde{\Gamma}^{\mx{T}}(P,Q)\right)(b_{t\mx{T}},c_{(1-t)\mx{T}},d_{(1-t)\mx{T}})dt\\
&=T (s-r)\int_0^1\tau\left[P'(b_{tT}c_{(1-t)T})Q'(b_{tT}d_{(1-t)T}))\right]dt\\
&=(s-r)\int_0^T\tau\left[P'(b_tc_{T-t})Q'(b_td_{T-t}))\right]dt,\end{align*}
and the two others cases are treated similarly.
\end{proof}

\appendix

\section{Appendix. The Intertwining Spaces $\PP(J)$ and $\C\{J\}$\label{appendix}}
Let $J$ be an index set. In this appendix, we describe the link between two spaces used to study trace polynomial functions, that is to say linear combination of functions $\M_N^J\to \M_N$ of the form
\[ \mx{A}\mapsto P_0(\mx{A})\tr(P_1(\mx{A}))\tr(P_2(\mx{A}))\cdots \tr(P_m(\mx{A})) \]
for some finite $m$, where $P_1,\ldots,P_m\in \C$ are noncommutative $*$-polynomials in $J$ variables. In Section~\ref{s.C{J}}, we already defined the space $\PP(J)$, introduced in~\cite{Driver2013,Kemp2013b}. Let us now define the space $\C\{J\}$, another space which has been introduced in~\cite{Cebron2013}. Finally, we will see that those two spaces are linked by a natural isomorphism.

The abstract trace polynomial algebra $\C\{J\}$ is a $\C$-algebra equipped with a center-valued expectation functional $\tr\colon\C\{J\}\to Z(\C\{J\})$: a linear map with values in the center of $\C\{J\}$ and satisfying $\tr(1_{\C\{J\}}) = 1_{{\C\{J\}}}$ and $\tr(\tr(A)B) = \tr(A)\tr(B)$ for all $A,B\in{\C\{J\}}$.  (Note: the symbol $\tr$ is presently denoting an abstract function, not necessarily the normalized trace on $\M_N$.)  The algebra ${\C\{J\}}$ is an extension of $\C\langle J\rangle = \C\langle X_j,X_j^\ast\colon j\in J\rangle$, the noncommutative polynomials in $J$ variables and their adjoints, in the sense that we have the injective inclusion $\C\langle J\rangle\subset \C\{J\}$. In~\cite{Cebron2013}, it is denoted by
\[  \C\{J\}\equiv \C\{X_j,X_j^{\ast}\colon j\in J\}.\]
It is defined by a universal property \cite[Universal Property 1.1]{Cebron2013}: let $\mathscr{A}$ be any $\C$-algebra equipped with a center-valued trace $\tau$, and specified elements $(A_{(j,\ex)})_{j\in J,\ex\in\{1,\ast\}}$ in $\mathscr{A}$.  Then there is a unique  algebra homomorphism $f\colon\C\{J\}\to\mathscr{A}$ such that
\begin{itemize}
\item[(1)] for all $(j,\ex)\in J\times\{ 1,\ast\}$, $f(X_j^\ex) = A_{(j,\ex)}$; and
\item[(2)] for all $X\in\C\{J\}$, $\tau(f(X)) = f(\tr(X))$.
\end{itemize}
This property uniquely defines $\C\{J\}$ up to adapted isomorphisms, cf.\ \cite[Proposition-Definition 1.3]{Cebron2013}, but we can also construct explicitly one realization of $\C\{J\}$ as a partially-symmetrized tensor algebra over $\C\langle J\rangle$. As a vector space, it has as a basis the set
\[ \left\{M_0\tr M_1\cdots \tr M_k, \quad k\in\N, \quad M_0,\ldots,M_k \text{ are monomials in } \C\langle J\rangle\right\} \]
Following~\cite{Cebron2013}, the universal property allows also to define a $\C\{J\}$-calculus. It is explicitly given as follows: for each $\mx{a}=(a_j)_{j\in J}\in\A^J$ and each $P_0,\ldots,P_k \in \C\langle J\rangle$, we have
$$\left(P_0\tr P_1\cdots \tr P_k\right) (\mx{a})=P_0(\mx{a})\cdot \tau (P_1(\mx{a}))\cdots \tau( P_k(\mx{a})),$$
and $P\mapsto P(\mx{a})$ is an algebra homomorphism. As a consequence, the space $\C\{J\}$ can be used to index the trace polynomial function on $\M_N$.

As we will now see, $\PP(J)$ is isomorphic to the ``scalar part'' $\tr(\C\{J\})$ of $\C\{J\}$.

\begin{lemma} For any index set $J$, there is an algebra isomorphism
\[ \Upsilon\colon \C\langle J\rangle \tensor \PP(J)\to   \C\{J\} \]
such that the restriction $\left.\Upsilon\right|_{1\tensor\PP(J)}$ is an algebra isomorphism onto $\tr(\C\{J\})$. More explicitly, $\Upsilon$ is given as follows: for any monomial $M_0\in\C\langle J\rangle$ and any words $\ex^{(j)}\in\EX$, we have
\[ \Upsilon(1)=1, \qquad \Upsilon\left(M_0\tensor v_{\ex^{(1)}}\cdots v_{\ex^{(k)}}\right) = M_0\tr(X_{\ex^{(1)}})\cdots \tr(X_{\ex^{(k)}}), \]
where, for all $\ex=((j_1,\ex_1),\ldots,(j_n,\ex_n))$, $X_\ex = X_{j_1}^{\ex_1}\cdots X_{j_n}^{\ex_n}$.  
\end{lemma}
\begin{proof}The homomorphism $\Upsilon$ transforms a basis of $\C\langle J\rangle \tensor \PP(J)$ into a basis of $\C\{J\}$, and is therefore a vector space isomorphism.  It is simple to check that it is also an algebra homomorphism. Alternatively, $\C\langle J\rangle\tensor\PP(J)$ is naturally isomorphic to the construction of $\C\{J\}$ in \cite[Appendix]{Cebron2013} as the partial symmetrization of the tensor algebra over $\C\langle J\rangle$ -- the polynomial algebra $\PP(J)$ is nothing other than the symmetric tensor algebra over $\C\langle J\rangle$. It is also easy to see that this map defines an algebra isomorphism using the universal property defining the space $\C\langle J\rangle$ in~\cite{Cebron2013}, where the center-valued expectation on the algebra $\C\langle J\rangle\tensor \PP(J)$ is the tracing map $\mathcal{T}$ of \cite[Definition 3.12]{Driver2013}, defined by
\[ \mathcal{T}(X_{\ex^{(0)}}\tensor v_{\ex^{(1)}}\cdots v_{\ex^{(k)}}) = v_{\ex^{(0)}}v_{\ex^{(1)}}\cdots v_{\ex^{(k)}}, \]
for any words $\ex^{(0)},\ldots,\ex^{(k)}\in\EX$.
\end{proof}

It is immediate that identifying $\tr(\C\{J\})$ with $\PP(J)$ via the isomorphism $\Upsilon$, the $\C\{J\}$-calculus is the same as the $\PP(J)$-calculus defined in~\ref{s.C{J}}: for all $P\in \tr(\C\{J\})\cong \PP(J)$, and all $\mx{a}\in \mathcal{A}^J$ we have
$$P(\mx{a})=[P]_{(\mathscr{A},\tau)}(\mx{a}).$$

\subsection*{Acknowledgments}  The authors wish to thank Bruce Driver for several very helpful mathematical conversations during the production of this paper.  In particular, we thank him for making us aware of the integral representation of Lemma \ref{l.Driver}, which greatly simplified the exposition of the proof of Theorem \ref{main}.  We also wish to extend our gratitude to the Fields Institute, for its excellent {\em Workshop on Analytic, Stochastic, and Operator Algebraic Aspects of Noncommutative Distributions and Free Probability} in July 2013, during which the authors met for the first time and conceived of the core ideas that led to the present work.

\bibliographystyle{acm}
\bibliography{Fluctuations}

\end{document}